\documentclass[reqno,11pt]{amsart}
\usepackage{amssymb}
\usepackage{latexsym,array}
\usepackage{amsfonts}
\newcommand{\C}{\mathbb {C}}

\newcommand{\cZ}{{\mathcal Z}}

\newcommand{\bl}{\boldsymbol{\mathfrak{e}_\ell}}
\newcommand{\br}{\boldsymbol{\mathfrak{e}_r}}
\newcommand{\bp}{\boldsymbol{\rho}}
\newcommand{\sbm}[1]{\left[\begin{smallmatrix} #1
                \end{smallmatrix}\right]}
\newcommand{\bm}{\mathbf {m}}
\newcommand{\cbm}[1]{\left(\begin{smallmatrix} #1
                \end{smallmatrix}\right)}

\newcommand{\cX}{\mathcal {X}}

\newcommand{\X}{\mathcal X_{_{[\alpha]}}}

\newtheorem{Pa}{Paper}[section]
\newtheorem{theorem}[Pa]{{\bf Theorem}}

\newtheorem{lemma}[Pa]{{\bf Lemma}}
\newtheorem{definition}[Pa]{{\bf Definition}}
\newtheorem{corollary}[Pa]{{\bf Corollary}}

\newtheorem{remark}[Pa]{{\bf Remark}}
\newtheorem{proposition}[Pa]{{\bf Proposition}}

\newtheorem{example}[Pa]{{\bf Example}}

\author[V. Bolotnikov]{Vladimir Bolotnikov}
\address{Department of Mathematics, College of William and Mary, Williamsburg, VA 23187-8795, USA}
\title[Lagrange interpolation]{Lagrange interpolation  
over division rings}
\begin{document}

\begin{abstract}
For a division ring $\mathbb F$, the polynomials $f\in\mathbb F$ can
be evaluated ``on the left" and ``on the right" giving rise to left and right
Lagrange interpolation problems. The problems containig interpolation conditions of the same 
type were considered in \cite{lam1} where the solvability criterion was given in terms 
of polynomial independence of interpolation nodes. We establish the solvability criterion 
and describe all solutions of low degree (less than the number of interpolation conditions imposed)
for the problem containing  both ``left" and ``right" conditions. 
\end{abstract}
\maketitle 

\section{Introduction}
\setcounter{equation}{0}
Given distinct nodes  $\alpha_1,\ldots,\alpha_n$ and target values $c_1,\ldots,c_n$ in a field $\mathbb F$, 
the classical Lagrange interpolation problem consists of finding a polynomial $f\in\mathbb F[z]$ such that
\begin{equation}
f(\alpha_i)=c_i\quad\mbox{for}\quad i=1,\ldots,n.
\label{1.1}
\end{equation}
If we consider $P_{n}(\mathbb F):=\{g\in\mathbb F[z]: \, \deg g< n\}$ and $\mathbb F^n$ as $n$-dimensional vector spaces over
$\mathbb F$, then the linear operator $T: \, P_{n}(\mathbb F)\to\mathbb F^n$ 
defined by $Tf=(f(\alpha_1),\ldots,f(\alpha_n))$ is injective, since a nonzero $f\in\mathbb F[z]$ 
cannot have more zeros in $\mathbb F$ than $\deg f$. Since $\dim P_{n}(\mathbb F)=\dim\mathbb F^n$, 
$T$ is also surjective, which leads us to the following observation.
\begin{remark} 
Given any distinct $\alpha_1,\ldots,\alpha_n$ and any $c_1,\ldots,c_n$ in a field $\mathbb F$,
there is a unique polynomial $f_L\in P_{n}(\mathbb F)$ subject to conditions \eqref{1.1}.
\label{R:1.0}
\end{remark}
The explicit formula for that unique $f_L$ (the {\em Lagrange interpolation formula}) 
\begin{equation}
f_L(z)=\sum_{i=1}^n \frac{p_i(z)c_i}{p_i(\alpha_i)},\quad \mbox{where}\quad p_i(z)=\prod_{j\neq i}(z-\alpha_j)
\quad\mbox{for}\quad i=1,\ldots,n,
\label{1.2}
\end{equation}
is easily verified as the $i$-th term $f_i(z)=\frac{p_i(z)c_i}{p_i(\alpha_i)}\in P_{n}(\mathbb F)$
satisfies conditions $f_i(\alpha_i)=c_i$ and  $f_i(\alpha_j)=0$ for all $j\neq i$.

\smallskip

We now turn to Lagrange interpolation over a {\em division ring} $\mathbb F$, which is different from the 
commutative case in two regards. First, left and right evaluation functionals $f\mapsto f^{\bl}(\alpha)$ and 
$f\mapsto f^{\br}(\alpha)$ on $\mathbb F[z]$ (see formulas \eqref{2.2} below)
give rise to two different (left and right) Lagrange interpolation problems: {\em given sets
\begin{equation}
\Lambda=\{\alpha_1,\ldots,\alpha_n\}\quad\mbox{and}\quad
 \Omega=\{\beta_1,\ldots,\beta_k\}
\label{1.7}
\end{equation}
of distinct elements in $\mathbb F$ along with given target values $c_1,\ldots,c_n$ and
$d_1,\ldots,d_k$ in $\mathbb F$, find a polynomial
$f\in\mathbb F[z]$ subject to left or right interpolation conditions
\begin{align}
f^{\bl}(\alpha_i)&=c_i\quad\mbox{for}\quad i=1,\ldots,n,\label{1.18}\\
f^{\br}(\beta_j)&=d_j\quad\mbox{for}\quad j=1,\ldots,k.\label{1.19}
\end{align}}
Evaluation functionals \eqref{2.2} also give rise to non-equivalent notions of left and right zeros;
consequently the solution sets of homogeneous problems \eqref{1.18} and \eqref{1.19} are respectively, 
the right ideal generated by the left minimal polynomial of $\Lambda$ and the left ideal generated by 
right minimal polynomial of $\Omega$.

\smallskip

Another distinction with the commutative case 
was indicated in \cite{gm}: as any polynomial $f\in\mathbb F[z]$ having two distinct left (right) zeros in the 
same conjugacy class of $\mathbb F$, actually has infinitely many zeros in this class, the target values 
cannot be assigned arbitrarily (at least a'priori) at more than two points within the same conjugacy class. 
The latter has been clarified in \cite{lam1} by introducing the notion of left (right) {\em polynomial 
independency} ($P$-independency; see Definition \ref{D:1.1}). Loosely speaking, a finite set $\Lambda\subset \mathbb F$ 
contains a maximal $P$-independent subset $\Lambda_0$, and the left values of each polynomial $f$ on 
$\Lambda_0$ uniquely determine $f^{\bl}$ on the whole $\Lambda$. This leads to consistency conditions 
for target values which either indicate that the problem \eqref{1.18} 
has no solutions, or allow us to disregard the interpolation conditions on  $\Lambda\backslash\Lambda_0$, hence making the 
problem \eqref{1.18} with left $P$-independent interpolation nodes a generic one. Similar observations hold true for
the right Lagrange problem \eqref{1.19}. The results concerning Lagrange problems with $P$-independent interpolation nodes
are the same (up to minor noncommutative adjustments) as in the commutative case. This material is briefly recalled in Section 2
in the form suitable for the subsequent analysis.

\smallskip

The main purpose of the present paper is to study the two-sided Lagrange problem that contains {\em both} 
left and right interpolation conditions \eqref{1.18}, \eqref{1.19}. We do not assume that the sets \eqref{1.7}
are disjoint, so left and right target values can be assigned to the same interpolation node. Without loss of generality,
we will assume that the sets $\Lambda$ and $\Omega$ in \eqref{1.7} are respectively left and right $P$-independent,
so that the left and right subproblems are consistent. Still, the combined problem may be inconsistent, 
and on the other hand, it may admit many low-degree solutions. In Section 3, we present the solvability criterion for 
the two-sided problem \eqref{1.18}, \eqref{1.19} to have a solution, establish a parametrization formula 
(which is fairly explicit under the assumption that the interpolation nodes are algebraic over the center of $\mathbb F$) 
producing all low-degree solutions. Two-sided polynomial independence and the two-sided Lagrange interpolation formula 
are also discussed in Section 3.

\section{Background}
\setcounter{equation}{0}

In what follows, $\mathbb F$ is assumed to be a {\em division ring} with the center $Z_{\mathbb F}$, and for  
each $\alpha\in\mathbb F$, we let $[\alpha]:=\{h\alpha h^{-1}: \, h\in\mathbb F\backslash\{0\}\}$ to denote its conjugacy class. 

\smallskip

We let $\mathbb F[z]$ to denote the ring of polynomials in one formal variable $z$ which commutes with 
coefficients from $\mathbb F$. Since the division algorithm holds in $\mathbb F[z]$ on
either side, any ideal (left or right) in $\mathbb F[z]$ is principal. We will write $\langle p\rangle_{\bf r}$
and $\langle p\rangle_{\boldsymbol\ell}$ for the right and the left ideal generated by $p\in\mathbb F[z]$
dropping the subscript if the ideal is two-sided. 
Any two-sided ideal of $\mathbb F[z]$ is generated by a polynomial with coefficients in $Z_\mathbb F$
(see e.g., \cite[Proposition 2.2.2]{cohn3}); the converse is clear since $Z_{\mathbb F[z]}=Z_{\mathbb F}[z]$.

\smallskip

The intersection of two left (right) ideals is a left (right)
ideal; the {\em least right} and {\em left common multiples} ${\bf lrcm}(f,g)$ and ${\bf llcm}(f,g)$
of two monic polynomials $f,g\in\mathbb F[z]$ are defined as generators of the respective ideals
\begin{equation}
\langle f\rangle_{\bf r}\cap\langle g\rangle_{\bf r}=
\langle{\bf lrcm}(f,g)\rangle_{\bf r}\quad\mbox{and}\quad
\langle f\rangle_{\boldsymbol\ell}\cap\langle g\rangle_{\boldsymbol\ell}=
\langle{\bf llcm}(f,g)\rangle_{\boldsymbol\ell}.
\label{2.1}
\end{equation}
\subsection{Evaluation functionals} 
Left and right evaluations of an $f\in\mathbb F[z]$ at $\alpha\in\mathbb F$ can be defined as 
the remainders of $f$ when divided by $\bp_\alpha(z)=z-\alpha$ on the left and on the 
right, respectively. As is easily verified, for any $\alpha\in\mathbb F$ and $f\in\mathbb F[z]$,
\begin{equation}
f(z)=f^{\bl}(\alpha)+\bp_\alpha\cdot(L_\alpha f)=f^{\br}(\alpha)+(R_\alpha f)(z)\cdot\bp_{\alpha}\qquad
(\bp_\alpha(z):=z-\alpha),
\label{2.1u}
\end{equation}
where $f^{\bl}(\alpha)$ and $f^{\br}(\alpha)$ are left and right evaluations of $f$ at $\alpha$:
\begin{equation}
f^{\bl}(\alpha)=\sum_{j=0}^m\alpha^j f_j\quad\mbox{and}\quad
f^{\br}(\alpha)=\sum_{j=0}^m f_j\alpha^j\quad\mbox{if}\quad f(z)=\sum_{j=0}^m z^j f_j,
\label{2.2}
\end{equation}
and where $L_\alpha f$ and $R_\alpha f$ are the polynomials given by
\begin{equation}
(L_\alpha f)(z)=\sum_{i+j=0}^{m-1}\alpha^if_{i+j+1}z^j,\quad
(R_\alpha f)(z)=\sum_{i+j=0}^{m-1}z^j f_{i+j+1}\alpha^i .
\label{2.3}
\end{equation}
\begin{remark}
{\rm The quantities introduced in \eqref{2.2} and \eqref{2.3} are related as follows: 
\begin{align}
(L_\alpha f)^{\br}(\beta)=
\sum_{k=0}^{\deg f-1}\sum_{j=0}^k \alpha^kf_{k+j+1}\beta^{k-j}&=(R_\beta f)^{\bl}(\alpha),\label{4.3}\\
\alpha \cdot(L_\alpha f)^{\br}(\beta)-(L_\alpha f)^{\br}(\beta)\cdot\beta&=f^{\bl}(\alpha)-f^{\br}(\beta),
\label{4.4}
\end{align}
for  any $\alpha,\beta\in\mathbb F$ and $f\in\mathbb F[z]$.
Indeed, equalities \eqref{4.3} are immediate from \eqref{2.3}, whhereas applying the right evaluation at $z=\beta$ to the first equality in 
\eqref{2.1u} gives
$$
f^{\br}(\beta)=f^{\bl}(\alpha)+(L_\alpha f)^{\br}(\beta)\cdot\beta-\alpha \cdot(L_\alpha f)^{\br}(\beta)
$$
which is equivalent to \eqref{4.4}.}
\label{R:4.4a}
\end{remark}
We next recall the product formulas for evaluations \eqref{2.2}. From the definitions \eqref{2.2}, one can see that
for any $f,g\in\mathbb F[z]$ and $\alpha\in\mathbb F$,
\begin{align}
(gf)^{\bl}(\alpha)&=\sum \alpha^kg^{\bl}(\alpha)f_k=\big(g^{\bl}(\alpha)\cdot f\big)^{\bl}(\alpha),\label{2.3a}\\
(gf)^{\br}(\alpha)&=\sum g_k f^{\br}(\alpha)\alpha^k=\big(g\cdot f^{\br}(\alpha)\big)^{\br}(\alpha),\label{2.3b}
\end{align}
which imply 
\begin{align}
(gf)^{\bl}(\alpha)&=\left\{\begin{array}{ccc}
g^{\bl}(\alpha)\cdot f^{\bl}\left(g^{\bl}(\alpha)^{-1}\alpha
g^{\bl}(\alpha)\right)&\mbox{if} &
g^{\bl}(\alpha)\neq 0, \\
0 & \mbox{if} & g^{\bl}(\alpha)= 0,\end{array}\right.
\label{2.6}\\
(gf)^{\br}(\alpha)&=\left\{\begin{array}{ccc} g^{\br}\left(f^{\br}(\alpha)\alpha
f^{\br}(\alpha)^{-1}\right)\cdot
f^{\br}(\alpha)&\mbox{if} & f^{\br}(\alpha)\neq 0, \\
0 & \mbox{if} & f^{\br}(\alpha)= 0.\end{array}\right.
\label{2.7}
\end{align}
Indeed, the top formula in \eqref{2.6} follows from \eqref{2.3a} and the computation
$$
\sum \alpha^kg^{\bl}(\alpha)f_k=g^{\bl}(\alpha)\sum (g^{\bl}(\alpha)^{-1}\alpha g^{\bl}(\alpha))^kf_k
=g^{\bl}(\alpha)\cdot f^{\bl}\left(g^{\bl}(\alpha)^{-1}\alpha
g^{\bl}(\alpha)\right).
$$
The top formula in \eqref{2.7} is justified similarly, while the bottom formulas in \eqref{2.6}, \eqref{2.7}
follow immediately from \eqref{2.3a}, \eqref{2.3b}. 
\begin{proposition}
For any $\alpha\in\mathbb F$ and $f,g\in\mathbb F[z]$,
\begin{equation}
L_\alpha(gf)=\left\{\begin{array}{lcc}
(L_\alpha g)\cdot f+g^{\bl}(\alpha)\cdot L_{\widetilde{\alpha}}f, &\mbox{if}& g^{\bl}(\alpha)\neq 0,\\
(L_\alpha g)\cdot f,&\mbox{if}& g^{\bl}(\alpha)= 0,\end{array}\right.
\label{Lprod}
\end{equation}
\label{P:lprod}
where $L_\alpha$ is defined as in \eqref{2.3} and where $\widetilde{\alpha}:=g^{\bl}(\alpha)^{-1}\alpha g^{\bl}(\alpha)$.
\end{proposition}
\begin{proof}
If $g^{\bl}(\alpha)=0$, then $gf=\bp_\alpha (L_\alpha g)\cdot f$, which proves the bottom formula in \eqref{Lprod}.
If $g^{\bl}(\alpha)\neq 0$, we define $\widetilde{\alpha}$ as above and observe that 
$\bp_{\alpha}g^{\bl}(\alpha)= g^{\bl}(\alpha)\bp_{\widetilde{\alpha}}$. Now we have, on account of \eqref{2.6},
\begin{align*}
\bp_{\alpha}\cdot L_\alpha(gf)=gf-(gf)^{\bl}(\alpha)&=(g-g^{\bl}(\alpha))\cdot f+
g^{\bl}(\alpha)(f-f^{\bl}(\widetilde{\alpha}))\\
&=\bp_{\alpha}\cdot (L_\alpha g)\cdot f+g^{\bl}(\alpha)\bp_{\widetilde{\alpha}}\cdot (L_{\widetilde{\alpha}}f)\\
&=\bp_{\alpha}\cdot (L_\alpha g)\cdot f+\bp_{\alpha}\cdot (L_{\widetilde{\alpha}}f),
\end{align*}
which completes the proof of \eqref{Lprod}.
\end{proof}
\subsection{Polynomial independence}
An element $\alpha\in\mathbb F$ is called a {\em left (right) zero} of $f\in\mathbb F[z]$ if
$f^{\bl}(\alpha)=0$ (respectively, $f^{\br}(\alpha)=0$). We will denote by $\cZ_{\boldsymbol\ell}(f)$ 
and $\cZ_{{\boldsymbol r}}(f)$ the respective sets of left and right zeros of $f$ 
and observe from \eqref{2.1u} that 
\begin{equation}
\alpha\in\cZ_{\boldsymbol\ell}(f) \; \; \Longleftrightarrow \; f\in\langle \bp_\alpha\rangle_{\bf r}
\quad\mbox{and}\quad  \alpha\in\cZ_{\boldsymbol r}(f) \; \; \Longleftrightarrow \; 
f\in \langle \bp_\alpha\rangle_{\boldsymbol\ell}.
\label{2.4}
\end{equation}
More generally, given an  algebraic set $\Delta\subset\mathbb F$, the polynomials
\begin{equation}
P_{\Delta,{\boldsymbol\ell}}={\bf lrcm}\left(\bp_{\alpha}: \, \alpha\in\Delta\right)\quad\mbox{and}
\quad P_{\Delta,{\bf r}}={\bf llcm}\left(\bp_{\alpha}: \, \alpha\in\Delta\right)
\label{minpol}
\end{equation}
generate the ideals $\langle P_{\Delta,{\boldsymbol\ell}}\rangle_{\bf r}$ and 
$\langle P_{\Delta,{\bf r}}\rangle_{\boldsymbol\ell}$ consisting of polynomials $f\in\mathbb F[z]$ such that 
$f^{\bl}\vert_\Delta=0$ and $f^{\br}\vert_\Delta=0$, respectively:
\begin{equation}
\Delta\subseteq\cZ_{\boldsymbol\ell}(f) \; \; \Longleftrightarrow \; 
f\in\langle P_{\Delta,{\boldsymbol\ell}}\rangle_{\bf r}
\quad\mbox{and}\quad  \Delta\subseteq\cZ_{\boldsymbol r}(f) \; \; 
\Longleftrightarrow \; f\in \langle P_{\Delta,{\bf r}}\rangle_{\boldsymbol\ell}.
\label{2.4u}
\end{equation}
The polynomials $P_{\Delta,{\boldsymbol\ell}}$ and $P_{\Delta,{\bf r}}$ are called {\em left} and 
{\em right minimal polynomials} of $\Delta$. In particular, it follows from \eqref{2.4u} that 
\begin{equation}
\Delta\subseteq\cZ_{\boldsymbol\ell}(P_{\Delta,{\boldsymbol\ell}})
\quad\mbox{and}\quad
\Delta\subseteq \cZ_{\bf r}(P_{\Delta,{\bf r}});
\label{minpol1}
\end{equation}
both inclusions can be proper, by Gordon-Motzkin theorem \cite{gm}. 
It is clear from \eqref{minpol1} that the numbers $\deg P_{\Delta,{\boldsymbol\ell}}$ and 
$\deg P_{\Delta,{\bf r}}$ cannot exceed the cardinality of $\Delta$.
\begin{definition}
{\rm A set $\Delta\subset\mathbb F$ is called} left polynomially independent {\rm if 
$\deg P_{\Delta,{\boldsymbol\ell}}=|\Delta|$, and it is called} 
right polynomially independent {\rm if $\deg P_{\Delta,{\bf r}}=|\Delta|$}.
\label{D:1.1}
\end{definition}
The notion of polynomial independence ($P$-independence) was introduced in \cite{lam1}; see also 
\cite{lamler1}, \cite{lamler2} for later elaborations. On account of \eqref{minpol1}, the equality 
$\deg P_{\Delta,{\boldsymbol\ell}}=|\Delta|$ means that the polynomials 
$\left(\bp_{\alpha}: \, \alpha\in\Delta\right)$ are left relatively prime, i.e., each one (say, $\bp_\beta$) 
is left coprime with the {\bf lrcm} of the others, i.e., with the left minimal polynomial 
$P_{\Delta\backslash\{\beta\},{\boldsymbol\ell}}$ of the set $\Delta\backslash\{\beta\}$. Since $\beta$ is 
the only zero of $\bp_\beta$, the latter simply means that 
$P_{\Delta\backslash\{\beta\},{\boldsymbol\ell}}^{\bl}(\beta)\neq 0$. We record this observation along 
with its right counter-part.
\begin{remark}
An algebraic set $\Delta\subset\mathbb F$ is left (right) $P$-independent if and only if
\begin{equation}
P_{\Delta\backslash\{\beta\},{\boldsymbol\ell}}^{\bl}(\beta)\neq 0\qquad (\mbox{respectively, \; } 
P_{\Delta\backslash\{\beta\},{\bf r}}^{\br}(\beta)\neq 0) 
\quad\mbox{for all}\quad \beta\in\Delta.
\label{pind}
\end{equation}
\label{R:1.1r}
\end{remark}
The following theorem characterizes $P$-independent sets in interpolation terms and provides two 
(left and right) noncommutative counter-parts of Remark \ref{R:1.0}.
\begin{theorem}
(1) A set $\Lambda=\{\alpha_1,\ldots,\alpha_n\}\subset\mathbb F$ is left $P$-independent if and only if
the left problem \eqref{1.18} has a solution in $P_n(\mathbb F)$ for any $c_1,\ldots, c_n\in\mathbb F$.
In this case, a unique $f_{\boldsymbol\ell}\in P_n(\mathbb F)$ subject to conditions \eqref{1.18} is given by the formula
\begin{equation}
f_{\boldsymbol\ell}(z)=\sum_{i=1}^n p_i(z)p_i^{\bl}(\alpha_i)^{-1}c_i,\; \; \mbox{where}\; p_i=P_{\Lambda\backslash\{\alpha_i\},{\boldsymbol\ell}}:=
{\bf lrcm}\big(\bp_{\alpha_j}: \, j\neq i\big).
\label{1.16}
\end{equation}
(2) A set $\Omega=\{\beta_1,\ldots,\beta_k\}\subset \mathbb F$ is right $P$-independent if and only if
the right problem \eqref{1.19} has a solution in $P_k(\mathbb F)$ for any $d_1,\ldots, d_k\in\mathbb F$.
In this case, a unique $f_{\bf r}\in P_k(\mathbb F)$ subject to conditions \eqref{1.19} is given by
\begin{equation}
f_{\bf r}(z)=\sum_{i=1}^k d_i q_i^{\br}(\beta_i)^{-1}q_i(z),\; \;\mbox{where}\; \; q_i=P_{\Lambda\backslash\{\beta_i\},{\bf r}}
:={\bf llcm}\big(\bp_{\beta_j}: \, j\neq i\big).
\label{1.16r}
\end{equation}
\label{T:1.1}
\end{theorem}
\begin{proof}
To argue as in the commutative case, we consider $P_n(\mathbb F)$ and $\mathbb F^n$ as right $\mathbb F$-modules over $\mathbb F$ and
define the right-linear operator $T: \, P_n(\mathbb F)\to \mathbb F^n$ by the formula
$Tf=(f^{\bl}(\alpha_1),\ldots,f^{\bl}(\alpha_n))$. Since $\dim_{_{\mathbb F}}P_n(\mathbb F)=\dim_{_{\mathbb F}}\mathbb F^n$, this operator is surjective
(i.e., the problem \eqref{1.18} has a solution in $P_n(\mathbb F)$ for any $c_1,\ldots, c_n\in\mathbb F$)
if and only if it is injective, i.e., {\em no nonzero polynomial of degree less than $n$ left-vanishes at $\Lambda$}.
The latter means that the left minimal polynomial $P_{\Lambda,{\boldsymbol\ell}}$ of $\Delta$ is of degree at least $n$, which means that
the set $\Lambda$ is left $P$-independent.

\smallskip

Conversely, if $\Lambda$ is left $P$-independent, then $P_{\Lambda\backslash\{\alpha_i\},{\boldsymbol\ell}}^{\bl}(\alpha_i)\neq 0$ for all $i=1,\ldots,n$, by Remark
\ref{R:1.1r}. Then the formula \eqref{1.16} makes sense and defines a polynomial $f_\ell\in P_n(\mathbb F)$ satisfying conditions
\eqref{1.18}. The uniqueness follows since the operator $T$ is injective. This completes the proof of part (1) of the theorem.
The proof of part (2) is similar once we consider
$P_k(\mathbb F)$ and $\mathbb F^k$ as left $\mathbb F$-modules over $\mathbb F$ and deal with the left linear operator
$T: \, P_k(\mathbb F)\to \mathbb F^k$ given by $Tf=(f^{\br}(\beta_1),\ldots,f^{\br}(\beta_k))$.
\end{proof}
\subsection{Consistency of interpolation conditions} By \eqref{2.4u}, the solution sets of homogeneous problems \eqref{1.18} and \eqref{1.19} are the ideals 
$\langle P_{\Lambda,{\boldsymbol\ell}}\rangle_{\bf r}$ and $\langle P_{\Omega,{\bf r}}\rangle_{\boldsymbol\ell}$. Combining the latter
Theorem \ref{T:1.1} leads us to the following conclusion.
\begin{remark}
If the sets $\Lambda$ and $\Omega$ in \eqref{1.7} are respectively, left and right $P$-independent, then 
all polynomials $f\in\mathbb F[z]$ satisfying conditions \eqref{1.18} and \eqref{1.19} are parametrized by respective formulas
\begin{equation}
f=f_{\boldsymbol\ell}+P_{\Lambda,{\boldsymbol\ell}}h\quad \mbox{and}\quad f=f_{\bf r}+gP_{\Omega,{\bf r}},\quad h,g\in\mathbb F[z]
\label{ap16}
\end{equation}
where $f_{\boldsymbol\ell}$ and $f_{\bf r}$ are defined in \eqref{1.16} and \eqref{1.16r}.
\label{R:1.3}
\end{remark}
Let us now consider the left interpolation problem  
\begin{equation}
f^{\bl}(\alpha_i)=c_i\quad\mbox{for}\quad i=1,\ldots,N
\label{ap1}
\end{equation}
where the set $\Delta=\{\alpha_1,\ldots,\alpha_N\}$ is not necessarily left $P$-independent. 
If $\deg P_{\Delta,{\boldsymbol\ell}}=n<N$, we can find a left $P$-independent subset $\Lambda\subset\Delta$ consisting of exactly $n$ elements
(a {\em left $P$-basis} of $\Delta$) and having the same left minimal polynomial as $\Delta$, that is, $P_{\Lambda,{\boldsymbol\ell}}=P_{\Delta,{\boldsymbol\ell}}$.
Without loss of generality we may let $\Lambda=\{\alpha_1,\ldots, \alpha_n\}$.

\smallskip

By Remark {R:1.3}, any polynomial $f$ satisfying conditions \eqref{ap1} (for $i=1,\ldots,n$) is of the form $f=f_{\boldsymbol\ell}+P_{\Delta,{\boldsymbol\ell}}h$
for some $h\in\mathbb F[z]$ and $f_{\boldsymbol\ell}$ given in \eqref{1.16}. Therefore
\begin{equation}
f^{\bl}(\gamma)=f_{\boldsymbol\ell}^{\bl}(\gamma)\quad \mbox{for all}\quad \gamma\in
\cZ_{\boldsymbol\ell}(P_{\Delta,{\boldsymbol\ell}}).
\label{ap20}
\end{equation}
By \eqref{minpol1}, we have in particular,  
$f^{\bl}(\alpha_j)=f_{\boldsymbol\ell}^{\bl}(\alpha_j)$ for $j=1,\ldots,N$.
Combining the latter equalities with \eqref{ap1} (for $j>n)$ and the formula \eqref{1.16} for $f_{\boldsymbol\ell}$, we get
\begin{equation}
\sum_{i=1}^n p_i^{\bl}(\alpha_j)p_i^{\bl}(\alpha_i)^{-1}c_i=c_j\quad\mbox{for}\quad j=n+1,\ldots,N.
\label{ap2}
\end{equation}
Thus, if the problem \eqref{ap1} is solvable, then the Lagrange polynomial $f_{\boldsymbol\ell}$ is a solution. For this to happen, the 
target value $c_j$ (for $j>n$) has to be equal to the actual value $f_{\boldsymbol\ell}^{\bl}(\alpha_j)$. If 
at least one of the equalities \eqref{ap2} fails, then the problem \eqref{ap1} is inconsistent. Otherwise, any polynomial
$f\in\mathbb F[z]$ satisfying the first $n$ conditions in \eqref{ap1} will satisfy the remaining conditions
automatically. After removing the redundant conditions we get a reduced interpolation problem based on the
left $P$-independent set $\Lambda$ and with the same solution set as the original problem. 

\smallskip

The same observations apply to the right-sided problem: a set $\Delta=\{\beta_1,\ldots, \beta_M\}$ with the right minimal 
polynomial  $P_{\Delta,{\bf r}}$ of degree $k$, can be rearranged so that its subset $\Omega=\{\beta_1,\ldots,\beta_k\}$ 
be right $P$-independent. Then the right interpolation problem 
\begin{equation}
f^{\br}(\beta_i)=d_i\quad\mbox{for}\quad i=1,\ldots,M
\label{ap3}
\end{equation}
has a solution if and only if the following compatibility conditions are satisfied
\begin{equation}
\sum_{i=1}^n d_iq_i^{\br}(\beta_i)^{-1}q_i^{\br}(\beta_j)=d_j\quad\mbox{for}\quad j=k+1,\ldots,M,
\label{ap4}
\end{equation}
where $q_i$ are given in \eqref{1.16r}. If the latter equalities hold true, then the last $M-k$ conditions in \eqref{ap3} are redundant and 
can be disregarded.
\begin{remark}
{\rm Since a set $\Delta$ is left (right) $P$-independent if (and clearly, only if) its intersection with each conjugacy class is, 
it suffices to verify compatibility conditions \eqref{ap2} and \eqref{ap4} within each conjugacy class
having non-empty intersection with $\Delta$. In other words, the problems \eqref{ap1} and \eqref{ap3} are solvable if their subproblems
within each conjugacy class are.}
\label{R:2.3}
\end{remark}
In the subsequent analysis, we will make frequent use of polynomials over $Z_{\mathbb F}$, for which the notions of
left and right values, and consequently, the notions of left and right zeros coincide (see formulas \eqref{2.2}).
Without any ambiguity, we may write $g(\alpha)$ and $\cZ(g)$ for the values and the
zero set of a central polynomial $g$. Besides, if $g\in Z_{\mathbb F}[z]$, then for each $\alpha$ and $\tau\neq 0$, we have
$g(\tau \alpha\tau^{-1})=\tau g(\alpha)\tau^{-1}$, so that $\cZ(g)$ contains with each $\alpha$ the whole conjugacy class $[\alpha]$.
\subsection{Extension formulas}\label{sb} The formula \eqref{ap20} shows that given an algebraic set $\Delta$ with a
fixed left $P$-basis $\Lambda\subset \Delta$ and given any $f\in\mathbb F[z]$, the Lagrange polynomial $f_{\boldsymbol\ell}$
constructed from the left values of $f$ on $\Lambda$ provides a unique extension of $f$ from $\Delta$ (even from $\Lambda$)
to a possibly larger set $\cZ_{\boldsymbol\ell}(P_{\Delta,{\boldsymbol\ell}})$
(the {\em left $P$-closure} of $\Delta$, in the terminology of \cite{lam1}). 
On the other hand, if $\gamma\not\in \cZ_{\boldsymbol\ell}(P_{\Delta,{\boldsymbol\ell}})$,
then the set $\Lambda\cup\{\gamma\}$ is left $P$-independent and the value $f^{\bl}(\gamma)$ is independent of
$f^{\bl}\vert_{\Lambda}$ (and therefore, of $f^{\bl}\vert_{\Delta}$), by Theorem \ref{T:1.1}.
Thus, it makes sense to consider extensions of polynomials within conjugacy classes. Similar observations apply to right evaluations.

\smallskip

If $V$ is an {\em algebraic} conjugacy class in $\mathbb F$, its left and right minimal polynomials \eqref{minpol}
are equal to the same central polynomial which will be denoted by $\mathcal X_{_V}$. Thus,
$\mathcal X_{_V}=P_{V,\boldsymbol\ell}=P_{V,{\bf r}}$ and $\cZ(\mathcal X_{_V})=V$. 

\smallskip

For any polynomial to be uniquely extended from a given $\Delta\subset V$ to the whole $V$, we need $\Delta$ to contain a 
left $P$-basis for $V$. Without loss of generality (and in order to use Lagrange interpolation formulas) we may assume that 
$\Delta$ itself is a left $P$-basis for $V$. In this case, the restriction of $f^{\bl}$ to a left $P$-basis of a conjugacy class $V$ uniquely determine not only
$f^{\bl}\vert_{_V}$ but also $f^{\br}\vert_{_V}$. Similarly.  the restriction of $f^{\br}$ to a right $P$-basis of $V$
uniquely determine $f^{\bl}\vert_{_V}$ and $f^{\br}\vert_{_V}$. Details are furnished below.
\begin{lemma}
Let $\Delta=\{\gamma_1,\ldots,\gamma_m\}$ be a left $P$-basis for the conjugacy class $V$.
Then for any $f\in\mathbb F[z]$ and $\gamma\in V$,
\begin{align}
f^{\bl}(\gamma)&=\sum_{i=1}^m 
P_{\Delta\backslash\{\gamma_i\},{\boldsymbol\ell}}^{\bl}(\gamma)P_{\Delta\backslash\{\gamma_i\},{\boldsymbol\ell}}^{\bl}(\gamma_i)^{-1}f^{\bl}(\gamma_i),\label{6.2r}\\
f^{\br}(\gamma)&=\sum_{i=1}^m \big(P_{\Delta\backslash\{\gamma_i\},{\boldsymbol\ell}}\cdot P_{\Delta\backslash\{\gamma_i\},{\boldsymbol\ell}}^{\bl}(\gamma_i)^{-1}
f^{\bl}(\gamma_i)\big)^{\br}(\gamma).
\label{6.3r}
\end{align}
\label{L:ext}
\end{lemma}
\begin{proof}
Since the set $\Delta$ is a left $P$-basis for $\Delta$, the formulas \eqref{6.2r}, \eqref{6.3r} make sense and besides,
$P_{\Delta,{\boldsymbol\ell}}=P_{V,{\boldsymbol\ell}}=\cX_{_V}$. Since the polynomial
$$
g(z)=f(z)-\sum_{i=1}^m P_{\Delta\backslash\{\gamma_i\},{\boldsymbol\ell}}(z)P_{\Delta\backslash\{\gamma_i\},{\boldsymbol\ell}}^{\bl}(\gamma_i)^{-1}f^{\bl}(\gamma_i)
$$
satisfies conditions $g^{\bl}(\gamma_i)=0$ for $i=1,\ldots,m$ (i.e., $g^{\bl}\vert_{\Delta}=0$), it follows from
\eqref{2.4u} that $g\in\langle P_{\Delta,{\boldsymbol\ell}}\rangle_{\boldsymbol\ell}=\langle \cX_{_V}\rangle$, and the latter
ideal is two-sided, since $\cX_{_V}\in Z_{\mathbb F}[z]$.
Then $g^{\bl}(\gamma)=g^{\br}(\gamma)=0$ for all $\gamma\in \cZ(\cX_{_V})=V$, which implies
\eqref{6.2r} and \eqref{6.3r}. 
\end{proof}
The right-sided version of Lemma \ref{L:ext} asserts that for a right $P$-basis $\Delta=\{\gamma_1,\ldots,\gamma_m\}$,
any polynomial $f\in\mathbb F$ and the right Lagrange polynomial constructed from $f^{\br}\vert_{\Delta}$
have the same left and right values at any $\gamma\in V$. We omit the precise formulation.
\begin{example}
{\rm If $\mathbb F=\mathbb H$, the skew field of real quaternions, any set $\Delta=\{\gamma_1,\gamma_2\}$ in a conjugacy class $V$
is a left and right $P$-basis for $V$. Adapting formulas \eqref{6.2r} and \eqref{6.3r} to this particular ``two-point" case, where 
$P_{\Delta\backslash\{\gamma_1\},{\boldsymbol\ell}}=\bp_{\gamma_2}$ and
$P_{\Delta\backslash\{\gamma_2\},{\boldsymbol\ell}}=\bp_{\gamma_1}$, we conclude:}
for any $f\in\mathbb H$ and any $\gamma_1,\gamma_2$ and $\gamma$ in the same conjugacy class,
\begin{align*}
f^{\bl}(\gamma)&=(\gamma-\gamma_2)(\gamma_1-\gamma_2)^{-1}f^{\bl}(\gamma_1)+(\gamma-\gamma_1)
(\gamma_2-\gamma_1)^{-1}f^{\bl}(\gamma_2),\\
f^{\br}(\gamma)&=(\gamma_1-\gamma_2)^{-1}f^{\bl}(\gamma_1)\gamma-
\gamma_2(\gamma_1-\gamma_2)^{-1}f^{\bl}(\gamma_1)\notag\\
&\qquad +\gamma_1(\gamma_1-\gamma_2)^{-1}f^{\bl}(\gamma_2)-(\gamma_1-\gamma_2)^{-1}f^{\bl}(\gamma_2)\gamma.
\end{align*}
{\rm Hence, the latter (well known) formulas turn out to be particular instances of Lagrange interpolation formula.}
\label{E:1.1}
\end{example}

\section{The two-sided problem}
\setcounter{equation}{0}
Assuming that the sets $\Lambda$ and $\Omega$ of interpolation nodes in \eqref{1.7} are respectively, left and right $P$-independent, i.e.,
such that
\begin{equation}
\deg P_{\Lambda,\boldsymbol\ell}=n\quad\mbox{and}\quad\deg P_{\Omega,{\bf r}}=k,
\label{ap17}
\end{equation}
we now address the two-sided problem \eqref{1.18}, \eqref{1.19}. This problem can be approached from two directions. First, one can start with 
the formula \eqref{ap16} describing all solutions to the left subproblem \eqref{1.18}, and then to characterize all parameters
$h\in\mathbb F[z]$ such that $f=f_{\boldsymbol\ell}+P_{\Lambda,{\boldsymbol\ell}}h$ satisfies right-sided conditions \eqref{1.19}.
The main difficulty here is that, according to \eqref{2.7},
$$
(P_{\Lambda,{\boldsymbol\ell}}h)^{\br}(\beta)=P_{\Lambda,{\boldsymbol\ell}}^{\br}\big(h^{\br}(\beta)\beta h^{\br}(\beta)^{-1}\big)\cdot h^{\br}(\beta),
$$ 
which does not allow us to separate $P_{\Lambda,{\boldsymbol\ell}}$ and $h$. Alternatively, we can start with more restricted (but simpler) problem 
by imposing extra interpolation conditions, and then to use the target values in these conditions as parameters describing solutions of the original problem.
More precisely, if the problem \eqref{1.18}, \eqref{1.19} admits a solution $f\in\mathbb F[z]$, then it follows from \eqref{4.4}
that the elements $\psi_{ij}=(L_{\alpha_i}f)^{\br}(\beta_j)$ satisfy equalities
\begin{equation}
\alpha_i \psi_{ij}-\psi_{ij}\beta_j=c_i-d_j\quad \mbox{for all}\quad i=1,\ldots, n; \; j=1,\ldots,k.
\label{5.3b}
\end{equation}
Note that equalities \eqref{5.3b} can be equivalently written as a single matrix equality
\begin{equation}
\sbm{\alpha_1 && 0 \\ &\ddots& \\ 0&& \alpha_n}X-X\sbm{\beta_1 && 0 \\ &\ddots& \\ 0&& \beta_k}
=\sbm{c_1 \\ \vdots \\ \\ c_n}\sbm{1 &  \cdots &  1}-\sbm{1 \\ \vdots \\ \\ 1}\sbm{d_1 &  \cdots &  d_k}
\label{ap19}
\end{equation}
satisfied by the matrix $X=[\psi_{ij}]\in\mathbb F^{n\times k}$.

\smallskip

We are going to use $\psi_{ij}$ as the prescribed target values for an unknown interpolant $f$, thus arriving at the following modified interpolation problem: 
{\em given two sets $\Lambda$ and $\Omega$ as in \eqref{1.7}
along with prescribed $c_i$, $d_j$, $\psi_{ij}\in\mathbb F$, find an $f\in\mathbb F[z]$ such that
\begin{equation}
f^{\bl}(\alpha_i)=c_i,\quad f^{\br}(\beta_j)=d_j\quad\mbox{and}\quad (L_{\alpha_i} f)^{\br}(\beta_j)=\psi_{ij}
\label{5.3a}
\end{equation}
for $i=1,\ldots,n$ and  $j=1,\ldots,k$.} 

\smallskip

This modified problem is quite simple: if the necessary conditions \eqref{5.3a} are met, the problem admits a unique solution in 
$P_{n+k}(\mathbb F)$, whereas the solution set of its homogeneous counter-part equals the product of ideals 
$\langle P_{\Lambda,\boldsymbol\ell}\rangle_{\bf r}$ and $\langle P_{\Omega,{\bf r}}\rangle_{\boldsymbol\ell}$.
Details are given in Propositions \ref{P:hom} and \ref{P:nhom} below.
\begin{proposition}
Given sets \eqref{1.7}, a polynomial $f\in\mathbb F[z]$ satisfies conditions
\begin{equation}
f^{\bl}(\alpha_i)=0,\quad f^{\br}(\beta_j)=0\quad\mbox{and}\quad
(L_{\alpha_i} f)^{\br}(\beta_j)=0
\label{5.3}
\end{equation}
for all $\alpha_i\in\Lambda$, $\beta_j\in\Omega$
if and only if it belongs to $\langle P_{\Lambda,\boldsymbol\ell}\rangle_{\bf r}\cdot \langle P_{\Omega,{\bf r}}\rangle_{\boldsymbol\ell}
= P_{\Lambda,\boldsymbol\ell}\cdot\mathbb F[z]\cdot P_{\Omega,{\bf r}}$.
\label{P:hom}
\end{proposition}
\begin{proof} For any $h\in\mathbb F[z]$, the polynomial $f=P_{\Lambda,\boldsymbol\ell}\cdot h\cdot P_{\Omega,\bf
r}$ satisfies conditions $f^{\bl}\vert_\Lambda=0$ and $f^{\br}\vert_{\Omega}=0$, by formulas \eqref{2.3a}, \eqref{2.3b}
and the definitions \eqref{minpol} of minimal polynomials. Since $P_{\Lambda,\boldsymbol\ell}^{\bl}(\alpha_i)=0$ 
for any $\alpha_i\in\Lambda$, we have
$$
L_{\alpha_i} f=L_{\alpha_i}(P_{\Lambda,\boldsymbol\ell}\cdot h\cdot P_{\Omega,\bf
r})=(L_{\alpha_i}P_{\Lambda,\boldsymbol\ell})\cdot h\cdot P_{\Omega,\bf r}
$$
and since $P_{\Omega,\bf r}^{\br}(\beta_j)=0$ for any $\beta_j\in\Omega$, 
evaluating the latter equality at $z=\beta_j$ on the right gives $(L_{\alpha_i} f)^{\br}(\beta_j)=0$.
Conversely, for fixed $\alpha_i$ and $\beta_j$, we have by \eqref{2.1u}, 
$$
f=f^{\bl}(\alpha_i)+\bp_{\alpha_i}\cdot(L_{\alpha_i} f)^{\br}(\beta_j)+\bp_{\alpha_i} \cdot (R_{\beta_j}L_{\alpha_i}f)\cdot\bp_{\beta_j}.
$$
If $f$ satisfies conditions \eqref{5.3}, we have $f=\bp_{\alpha_i}h \bp_{\beta_j}$ with $h=R_{\beta_j}L_{\alpha_i}f$. Hence, 
$f$ belongs to $\langle \bp_{\alpha_i}\rangle_{\bf r}
\cdot \langle\bp_{\beta_j}\rangle_{\boldsymbol\ell}$ for all
$(\alpha_i,\beta_j)\in\Lambda\times \Omega$. By \eqref{minpol}, it then follows that for each fixed $\beta_j$, $R_{\beta_j} f$ belongs to 
$\langle P_{\Lambda,\boldsymbol\ell}\rangle_{\bf r}$ so that $f$ belongs to 
$\langle P_{\Lambda,\boldsymbol\ell}\rangle_{\bf r}\cdot \langle\bp_{\beta_j}\rangle_{\boldsymbol\ell}$ for all $\beta_j\in\Omega$.
Using the same argument as above we come the desired conclusion.
\end{proof}
\begin{remark}
{\rm Conditions \eqref{5.3a} are not independent: it follows from \eqref{4.4} that after dropping left (or right) conditions 
in \eqref{5.3b}, the remaining conditions still define the product-ideal $P_{\Lambda,\boldsymbol\ell}\cdot\mathbb F[z]\cdot P_{\Omega,{\bf r}}$. 
It is of some interest to characterize the latter set in terms of (presumably, $n+k$) independent conditions. One possible choice is to take 
all left conditions in \eqref{5.3b} and certain $k$ linear combinations of the two-sided conditions. In more detail, if $[v_1 \; v_2 \; \ldots \; v_n]$ 
denote the bottom row of the matrix $W^{-1}$, where $W=\big[\alpha_i^{j-1}\big]_{i,j=1}^n$
is the Vandermonde matrix associated with $\Lambda$ (it is invertible since $\Lambda$ is left $P$-independent; see \cite{lam1}), then, whenever 
a polynomial $f$ satisfies conditions
$$
f^{\bl}(\alpha_i)=0 \; \; (i=1,\ldots,n)\quad\mbox{and}\quad \sum_{i=1}^n v_i L_{\alpha_i}^{\br}(\beta_j)=0 \; \; (j=1,\ldots,k),
$$
it belongs to $P_{\Lambda,\boldsymbol\ell}\cdot\mathbb F[z]\cdot P_{\Omega,{\bf r}}$.}
\label{R:3.31}
\end{remark}
\begin{proposition}
Under the assumptions \eqref{ap17}, the equalities \eqref{5.3b} are necessary and sufficient for 
the problem \eqref{5.3a} to have a solution. In this case, the formula 
\begin{equation}
f(z)=\sum_{i=1}^n p_i(z)p_i^{\bl}(\alpha_i)^{-1}c_i+P_{\Lambda,\boldsymbol\ell}(z)\cdot\sum_{i=1}^n\sum_{j=1}^k 
p_i^{\bl}(\alpha_i)^{-1}\psi_{ij}q_j^{\br}(\beta_j)^{-1}q_j(z),
\label{bap1}
\end{equation}
where $p_i=P_{\Lambda\backslash\{\alpha_i\},\boldsymbol\ell}$ and $q_j=P_{\Omega\backslash\{\beta_j\},{\bf r}}$, 
defines a unique polynomial $f\in P_{n+k}(\mathbb F)$ satisfying conditions \eqref{5.3a}.
\label{P:nhom}
\end{proposition}
\begin{proof} 
The necessity of \eqref{5.3b} follows from equality \eqref{4.4}. The polynomial $f$ in \eqref{bap1} is of the form 
$f=f_{\boldsymbol\ell}+P_{\Lambda,{\boldsymbol\ell}}h$ (where $f_{\boldsymbol\ell}$ is the left Lagrange polynomial \eqref{1.16}) and hence,
it satisfies the left-sided conditions in \eqref{5.3}, by Remark \ref{R:1.3}. It remains to show that if equalities \eqref{5.3b} hold, the 
polynomial\eqref{bap1} satisfies the rest of conditions in \eqref{5.3a}. Once the two-sided conditions in \eqref{5.3a} will be confirmed, the 
right-sided conditions will follow automatically, by \eqref{4.4} and \eqref{5.3b}:
$$
f^{\br}(\beta_j)=f^{\bl}(\alpha_i)-\alpha_i \cdot(L_{\alpha_i} f)^{\br}(\beta_j)+(L_{\alpha_i} f)^{\br}(\beta_j)=c_i-\alpha_i\psi_{ij}+\psi_{ij}\beta_j=d_j,
$$
for $j=1,\ldots,k$. To verify that $f$ satisfies the third condition in \eqref{5.3a} we first note that 
\begin{equation}
P_{\Lambda,\boldsymbol\ell}=p_i\cdot \bp_{\widetilde{\alpha}_i},\quad\mbox{where}\quad
\widetilde{\alpha}_i=p_i^{\bl}(\alpha_i)^{-1}\alpha_i p_i^{\bl}(\alpha_i)
\label{5.3c}
\end{equation}
for every $i=1,\ldots,n$. Indeed, the polynomial $g=p_i\cdot \bp_{\widetilde{\alpha}_i}$ satisfies 
$g^{\bl}(\alpha_j)=0$ for $j\neq i$ (by the definition of $p_i=P_{\Lambda\backslash\{\alpha_i\},\boldsymbol\ell}$) and for $j=i$ (by the formula \eqref{2.6}).
Since $\deg g=n$, it follows that $g$ is the minimal polynomial of $\Lambda$, i.e., that $g=P_{\Lambda,\boldsymbol\ell}$. 
By Proposition \ref{P:lprod}, we now have, for each $i,s=1,\ldots,n$, 
\begin{equation}
L_{\alpha_s}P_{\Lambda,\boldsymbol\ell}=L_{\alpha_s}(p_i\cdot \bp_{\widetilde{\alpha}_i})=
(L_{\alpha_s}p_i)\cdot \bp_{\widetilde{\alpha}_i}+p_i^{\bl}(\alpha_s).
\label{Lproda}
\end{equation}
We also observe the identity
\begin{equation}
p_1(z)p_1^{\bl}(\alpha_1)^{-1}+p_2(z)p_2^{\bl}(\alpha_2)^{-1}+\ldots +p_n(z)p_n^{\bl}(\alpha_n)^{-1}- 1\equiv 0.
\label{5.3d}
\end{equation}
Indeed, the polynomial on the left side is of degree less than $n$ and has left zeros at $\alpha_1,\ldots,\alpha_n$. Since 
the set $\Lambda$ is left $P$-independent, \eqref{5.3d} follows. In particular, we conclude from \eqref{5.3d} that 
\begin{equation}
L_{\alpha}\bigg(\sum_{i=1}^n p_i\cdot p_i^{\bl}(\alpha_i)^{-1}\bigg)=0\quad \mbox{for any} \quad \alpha\in\mathbb F. 
\label{5.3e}
\end{equation}
For any fixed $\alpha_s\in\Lambda$ and $\beta_t\in\Omega$, we have for 
$f$ of the form \eqref{bap1},
\begin{equation}
(L_{\alpha_s}f)^{\br}(\beta)=\sum_{i=1}^n \big((L_{\alpha_s}p_i)\cdot p_i^{\bl}(\alpha_i)^{-1}c_i\big)^{\br}(\beta_t)+\Phi
\label{ap2a}
\end{equation}
where we have set for short,
$$
\Phi=\sum_{i=1}^n\sum_{j=1}^k \big (L_{\alpha_s}\big(P_{\Lambda,\boldsymbol\ell}\cdot p_i^{\bl}(\alpha_i)^{-1}\psi_{ij}q_j^{\br}(\beta_j)^{-1}q_j\big)\big)^{\br}(\beta_t).
$$
Due to equalities $P_{\Lambda,\boldsymbol\ell}^{\bl}(\alpha_s)=0$ and $q_j^{\br}(\beta_t)=P_{\Omega\backslash\{\beta_j\},{\bf r}}^{\br}(\beta_t)=0$ 
(for all $j\neq t$), the latter expression for $\Phi$ can be written as
$$
\Phi=\sum_{i=1}^n \big (\big(L_{\alpha_s}P_{\Lambda,\boldsymbol\ell}\big)\cdot p_i^{\bl}(\alpha_i)^{-1}\psi_{it}\big)^{\br}(\beta_t).
$$
Substituting \eqref{Lproda} into the latter equality and taking into account that $p_i^{\bl}(\alpha_s)=0$ for all $i\neq s$, we have
\begin{align}
\Phi&=\sum_{i=1}^n \big (\big((L_{\alpha_s}p_i)\cdot \bp_{\widetilde{\alpha}_i}+p_i^{\bl}(\alpha_s)\big)\cdot p_i^{\bl}(\alpha_i)^{-1}
\psi_{it}\big)^{\br}(\beta_t)\notag\\
&=\sum_{i=1}^n \big ((L_{\alpha_s}p_i)\cdot \bp_{\widetilde{\alpha}_i}\cdot p_i^{\bl}(\alpha_i)^{-1}\psi_{it}\big)^{\br}(\beta_t)
+\psi_{st}\notag\\
&=\sum_{i=1}^n \big ((L_{\alpha_s}p_i)p_i^{\bl}(\alpha_i)^{-1}\cdot \bp_{\alpha_i} \cdot\psi_{it}\big)^{\br}(\beta_t)+\psi_{st},
\label{ap3a}
\end{align}
where the last equality holds, by the definition of $\widetilde{\alpha}_i$ in \eqref{5.3c}. Since by \eqref{5.3b},
$$
\big(\bp_{\alpha_i}\cdot \psi_{it}\big)^{\br}(\beta_t)=\psi_{it}\beta_t-\alpha_i\psi_{it}=d_t-c_i,
$$
we may invoke formula \eqref{2.3b} to write \eqref{ap3a} as 
\begin{align*}
\Phi&=\sum_{i=1}^n \big ((L_{\alpha_s}p_i)p_i^{\bl}(\alpha_i)^{-1}\cdot (\bp_{\alpha_i} \cdot\psi_{it})^{\br}(\beta_t)\big)^{\br}(\beta_t)+\psi_{st}\\
&=\sum_{i=1}^n \big ((L_{\alpha_s}p_i)p_i^{\bl}(\alpha_i)^{-1}\cdot (d_t-c_i)\big)^{\br}(\beta_t)+\psi_{st}.
\end{align*}
Substituting the latter expression for $\Phi$ into \eqref{ap2a} and making use of \eqref{5.3d} we get
\begin{align*}  
(L_{\alpha_s}f)^{\br}(\beta_t)&=\sum_{i=1}^n \big ((L_{\alpha_s}p_i)p_i^{\bl}(\alpha_i)^{-1}\cdot d_t\big)^{\br}(\beta_t)+\psi_{st}\\
&=\bigg(L_{\alpha_s}\bigg(\sum_{i=1}^n p_i\cdot p_i^{\bl}(\alpha_i)^{-1}\bigg)\cdot d_t\bigg)^{\br}(\beta_t)+\psi_{st}=\psi_{st}.
\end{align*}
The difference of two polynomials $f,g$ of degree less than $n+k$ and  satisfying conditions \eqref{5.3a} is in
$P_{\Lambda,\boldsymbol\ell}\cdot\mathbb F[z]\cdot P_{\Omega,{\bf r}}$, by Proposition \ref{P:hom}. Therefore $f=g$ and the 
uniqueness of $f\in P_{n+k}(\mathbb F)$ satisfying conditions \eqref{5.3a} follows.
\end{proof}
\begin{remark}
{\rm The formula \eqref{bap1} looks asymmetric with respect to the left and right interpolation subproblems. To dismiss this asymmetry, note 
that the polynomial $f$ in \eqref{bap1} can be alternatively written in terms of the right Lagrange polynomial \eqref{1.16r} as 
\begin{equation}
f(z)=\sum_{j=1}^k d_j q_j^{\br}(\beta_j)^{-1}q_j(z)+\sum_{i=1}^n\sum_{j=1}^k p_i(z)p_i^{\bl}(\alpha_i)^{-1}\psi_{ij}q_j^{\br}(\beta_j)^{-1}P_{\Omega,{\bf r}}(z).
\label{ap1a}
\end{equation}
Verification of equality of right-hand side expressions in formulas \eqref{bap1} and \eqref{ap1a} relies on relations \eqref{5.3b} and is quite straightforward.}
\label{R:3.30}
\end{remark}
Upon interpreting the target values $\psi_{ij}\in\mathbb F$ in \eqref{5.3a} as unspecified parameters subject to {\em Sylvester equations} \eqref{5.3b}
we arrive at the following consequence of Proposition \ref{P:nhom}
\begin{proposition}
Under assumptions \eqref{ap17}, the formula \eqref{bap1} (or \eqref{ap1a}) establishes a bijection between $nk$-tuples $\{\psi_{ij}\}$ of solutions 
to the Sylvester equations \eqref{5.3b} (equivalently, solutions $X=[\psi_{ij}]\in\mathbb F^{n\times k}$ to the
matrix Sylvester equation \eqref{ap19}) and all polynomials $f\in P_{n+k}(\mathbb F)$ satisfying conditions \eqref{1.18}, \eqref{1.19}.
\label{P:nhoma}
\end{proposition}
Consequently, the problem \eqref{1.18}, \eqref{1.19} has a solution (a unique solution in $P_{n+k}(\mathbb F)$)
if and only if each equation in \eqref{ap19} has a solution in $\mathbb F$ (respectively, has a unique solution in $\mathbb F$). To proceed 
further, we recall some needed results concerning the solvability in $\mathbb F$ of the scalar Sylvester equation
\begin{equation}
\alpha x-x\beta=\gamma, \qquad \alpha,\beta,\gamma\in\mathbb F;
\label{4.25f}
\end{equation}
the study of the latter in the context of general division rings goes back to \cite{jacob1} and \cite{johnson}.
(see also \cite{cohn2}, \cite[Section 6]{lamler2}) and to Hamilton (see e.g., \cite[p. 123]{tait}) in the case of real quaternions. 
We will assume that $\alpha$ is algebraic over $Z_{\mathbb F}$. In this case, the conjugacy class $[\alpha]$ is algebraic and its 
minimal polynomial $\X\in Z_{\mathbb F}[z]$ turns out to be the {\em minimal central polynomial} for $\alpha$ (as well as for any 
$\beta\in[\alpha]$). Furthermore,
\begin{align}
\X=\bp_\beta \cdot (L_\beta \X)&=(R_\beta \X)\cdot\bp_\beta,\quad L_\beta \X=R_\beta \X,\label{ma11}\\
\X^\prime(\beta)&=(L_\beta \X)^{\bl}(\beta)\neq 0\quad\mbox{for each}\quad \beta\in[\alpha].
\label{ma10}
\end{align}
Indeed, the first equality in \eqref{ma11} holds since each $\beta\in [\alpha]$ is a left and a right zero of $\X$, and the second equality holds since
$\X\in Z_{\mathbb F}[z]$. The chain rule gives $\X^\prime=L_\beta \X+\bp_\beta \cdot (L_\beta \X)^\prime$, which being
evaluated at $\alpha$ on the left, implies the equality in \eqref{ma10}. Since the formal derivative $\X^\prime$ also belongs to $Z_{\mathbb F}[z]$ and
$\deg \X^\prime<\deg \X$, it follows that $\X^\prime(\beta)\neq 0$ for any $\beta\in[\alpha]$.

\smallskip

Given a triple $(\alpha,\beta;\gamma)$ with $\alpha$ algebraic and $\deg \X=\kappa$, the element
\begin{equation}
\Psi_{\alpha,\beta}(\gamma):=\left\{\begin{array}{cc}
-\big(L_\alpha \X\gamma\big)^{\br}(\beta)\cdot \X(\beta)^{-1}, & \mbox{if} \; \beta\not\in[\alpha],\\
{\displaystyle\sum_{j=1}^{\kappa-1}\sum_{i=0}^{j-1}\frac{(-1)^{i+j}}{(j+1)!}\cbm{j-1 \\ i}
\alpha^i\gamma\X^{(j+1)}(\beta)\beta^{j-i-1}\cdot \X^\prime(\beta)^{-1}},& \mbox{if} \; \beta\in[\alpha],
\end{array}\right.
\label{combo}
\end{equation}
is well defined, due to \eqref{ma10} and since $\cZ(\X)=[\alpha]$ (and hence, $\X(\beta)\neq 0$ for $\beta\not\in[\alpha]$).
\begin{proposition}
Let $\alpha\in\mathbb F$ be algebraic and let $\X\in Z_{\mathbb F}[z]$ be its minimal polynomial.

\smallskip
\noindent
{\rm (1)} If $\beta\not\in[\alpha]$, then for any $\gamma\in\mathbb F$, the equation \eqref{4.25f} has a unique solution 
$x\in\mathbb F$, given by the top formula in \eqref{combo}: $x=\Psi_{\alpha,\beta}(\gamma)$.

\smallskip
\noindent
{\rm (2)} If $\beta\in[\alpha]$, then the equation \eqref{4.25f} has a solution in $\mathbb F$ if and only if
\begin{equation}
\big(L_\alpha \X\gamma\big)^{\br}(\beta)=0.
\label{4.25ha}
\end{equation}
In this case, all solutions $x\in\mathbb F$ to the equation \eqref{4.25f} are given by 
\begin{equation}
x=\Psi_{\alpha,\beta}(\gamma)+\varphi,
\label{gsyl}
\end{equation}
where $\Psi_{\alpha,\beta}(\gamma)$ is defined by the bottom formula in \eqref{combo} and 
$\varphi$ is any intertwiner of $\alpha$ and $\beta$ (i.e., $\alpha\varphi=\varphi\beta_j$).
\label{P:klj}
\end{proposition}
\begin{proof}[Proof of (1)] To verify that $x=\Psi_{\alpha,\beta}(\gamma)$ of the form \eqref{combo} solves the equation
\eqref{4.25f}, we use \eqref{4.4} (with $f=\X\gamma$)
and relations $\beta \X(\beta)=\X(\beta)\beta$ and $\X(\alpha)=0$:
\begin{align*}
(\alpha x-x\beta)\X(\beta)&=-
\alpha(L_\alpha \X\gamma)^{\br}(\beta)+(L_\alpha \X\gamma)^{\br}(\beta)\beta\\
&=(\X\gamma)^{\br}(\beta)-(\X\gamma)^{\bl}(\alpha)=\gamma\X(\beta),
\end{align*}
which implies \eqref{4.25f}, as $\X(\beta)\neq 0$.
For the uniqueness part, let $x$ be any solution to the equation \eqref{4.25f}, or equivalently,
to the equation 	
$\; x\bp_\beta-\bp_\alpha x=\gamma$. Multiplying both sides in the latter equality by $L_\alpha \X$
on the left we get, on account of \eqref{ma11},
$$
(L_\alpha \X) x\bp_\beta-\X x=(L_\alpha\X)\gamma,
$$
which being evaluated at $\beta$ on the right gives (since $\X\in Z_{\mathbb F}[z]$)
\begin{equation}
-x\X(\beta)=\big((L_\alpha \X)\gamma\big)^{\br}(\beta)
\label{klj1}
\end{equation}
which uniquely defines $x$ via the top formula in \eqref{combo}.
\end{proof}
\begin{proof}[Proof of (2)] By part (1), equality \eqref{klj1} holds
for any solution $x$ (if exists) to the equation \eqref{4.25f}.
If $\beta\in[\alpha]$, then $\X(\beta)=0$ and \eqref{klj1} amounts to \eqref{4.25ha},
which completes the proof of the ``only if" part.
To prove the ``if" part, we start with the general formula
$$
(L_\alpha f)(z)=\sum_{j=0}^{\deg f-1}\frac{(-1)^{j}}{(j+1)!}\, (z-\alpha)^{j}f^{(j+1)}(z)
$$
relating the left backward shift $L_\alpha$ of a polynomial $f$ with its formal derivatives.
Applying the latter formula to $f=\X \gamma$ gives
$$
L_\alpha (\X\gamma)=\sum_{j=0}^{\kappa-1}\frac{(-1)^{j}}{(j+1)!}\, \bp_\alpha^j\X^{(j+1)}\gamma,\quad\mbox{where}\quad \kappa=\deg\X.
$$
Assuming that \eqref{4.25ha} is in force, we evaluate both sides of the last equality at $\beta$ on the right and arrive at
\begin{align}
0&=\sum_{j=0}^{\kappa-1}\frac{(-1)^{j}}{(j+1)!}\big(\bp_\alpha^j\X^{(j+1)}\gamma\big)^{\br}(\beta)\notag\\
&=(\X^\prime\gamma)^{\br}(\beta)+
\sum_{j=1}^{\kappa-1}\frac{(-1)^{j}}{(j+1)!}\big(\bp_\alpha \bp_\alpha^{j-1}\X^{(j+1)}\gamma\big)^{\br}(\beta).
\label{klj3}
\end{align}
Since $\X^{(j)}\in Z_{\mathbb F}[z]$ for all $j\ge 0$ and since
$$
(\bp_\alpha \,  f)^{\br}(\beta)=f^{\br}(\beta)\beta-\alpha f^{\br}(\beta)\quad\mbox{for all}\quad f\in\mathbb F[z],
$$
we can write \eqref{klj3} equivalently as
$$
\gamma\X^\prime(\beta)=\sum_{j=1}^{\kappa-1}\frac{(-1)^{j}}{(j+1)!}\left(\alpha\big(\bp_\alpha^{j-1}\X^{(j+1)}\gamma\big)^{\br}(\beta)
-\big(\bp_\alpha^{j-1}\X^{(j+1)}\gamma\big)^{\br}(\beta)\beta\right).
$$
Since $\X^\prime(\beta)\neq 0$ and $\X^\prime(\beta)\beta=\beta \X^\prime(\beta)$,
we can divide both sides of the last equality by $\X^\prime(\beta)$
on the right and write the resulting equality as $\gamma=\alpha x_0-x_0\beta$, where
\begin{equation}
x_0=\sum_{j=1}^{\kappa-1}\frac{(-1)^{j}}{(j+1)!}\big(\bp_\alpha^{j-1}\X^{(j+1)}\gamma\big)^{\br}(\beta)\cdot \X^\prime(\beta)^{-1}.
\label{klm}
\end{equation}
The latter means that $x_0$ is a solution to the equation \eqref{4.25f}. A  more detailed formula for $x_0$ as in \eqref{combo}
follows upon plugging in the equalities
$$
\big(\bp_\alpha^{j-1}\X^{(j+1)}\gamma\big)^{\br}(\beta)=\sum_{i=0}^{j-1}(-1)^{i}\cbm{j-1 \\ i}
\alpha^i\gamma \X^{(j+1)}(\beta)\beta^{j-i-1}.
$$
into the right side of \eqref{klm}. Combining $x_0$ with 
the general solution $\varphi$ of the homogeneous Sylvester equation $\alpha x-x\beta=0$ gives \eqref{gsyl}.
\end{proof}
\begin{remark}
{\rm If $\beta$ is algebraic and $\alpha\not\sim\beta$, one can multiply the identity
$\; x\bp_\beta-\bp_\alpha x=\gamma$ by $R_\beta \mathcal X_{_{[\beta]}}=L_\beta \mathcal X_{_{[\beta]}}$ on the right
and evaluate the resulting identity at $\alpha$ on the left to get an alternative formula for $\Psi_{\alpha,\beta}(\gamma)$ in case
$\alpha\not\in\beta$.}
\label{R:alt}
\end{remark}
\begin{remark}
{\rm The case where $\alpha$ and $\beta$ are both transcendental is more subtle. An example in \cite{berg} shows that
even with $\alpha\not\sim\beta$, the equation \eqref{4.25f} may have no solutions. We are not aware
of explicit solvability or uniqueness criteria for the transcendental case. For this reason, our further results on
the two-sided problem \eqref{1.18}, \eqref{1.19} are established under the (certainly restrictive) assumption that either 
all left or all right interpolation nodes are algebraic.}
\label{R:alta}
\end{remark}
\begin{theorem}
Let us assume that the set $\Lambda=\{\alpha_1,\ldots,\alpha_n\}$ is algebraic over $Z_{\mathbb F}$ and left $P$-independent, whereas the set $\Omega=\{\beta_1,\ldots,\beta_k\}$
is right $P$-independent. The two-sided Lagrange problem \eqref{1.18}, \eqref{1.19} has a solution if and only if 
\begin{equation}
\big(L_{\alpha_i}\cX_{[\alpha_i]}c_i\big)^{\br}(\beta_j)=\big(L_{\alpha_i}\cX_{[\alpha_i]}d_j\big)^{\br}(\beta_j), \quad \mbox{whenever}\quad \alpha_i\sim\beta_j.
\label{4.25hab}
\end{equation}
In this case, all $f\in P_{n+k}(\mathbb F)$ satisfying conditions \eqref{1.18}, \eqref{1.19} are given by the formula 
\begin{align}
f(z)=\sum_{i=1}^n p_i(z)p_i^{\bl}(\alpha_i)^{-1}c_i&+P_{\Lambda,\boldsymbol\ell}(z)\cdot\sum_{i=1}^n\sum_{j=1}^k
p_i^{\bl}(\alpha_i)^{-1}\Psi_{\alpha_i,\beta_j}(c_i-d_j)q_j^{\br}(\beta_j)^{-1}q_j(z)\notag\\
&+P_{\Lambda,\boldsymbol\ell}(z)\cdot\sum_{\alpha_i\sim\beta_j}
p_i^{\bl}(\alpha_i)^{-1}\varphi_{ij}q_j^{\br}(\beta_j)^{-1}q_j(z),
\label{brap1}
\end{align}
where  $p_i=P_{\Lambda\backslash\{\alpha_i\},\boldsymbol\ell}$, $q_j=P_{\Omega\backslash\{\beta_j\},{\bf r}}$, where $\Psi_{\alpha_i,\beta_j}(c_i-d_j)$ 
are defined via formulas \eqref{combo}, and where 
$\varphi_{ij}$ is any intertwiner of $\alpha_i$ and $\beta_j$ (i.e., $\alpha_i\varphi_{ij}=\varphi_{ij}\beta_j$).
\label{T:4.1}
\end{theorem}
\begin{proof}
By Proposition \ref{P:nhoma}, the problem \eqref{1.18}, \eqref{1.19} has a solution if and only if each Sylvester equation in \eqref{5.3b} is solvable.
This is the case for each non-conjugate pair $\alpha_i\not\sim\beta_j$, by part (1) in Proposition \ref{P:klj}. If $\alpha_i\sim\beta_j$, then the corresponding 
Sylvester equation in \eqref{5.3b} has a solution if and only if \eqref{4.25ha} holds with $\alpha=\alpha_i$, $\beta=\beta_j$ and $\gamma=c_i-d_j$, that is,
$$
\big(L_{\alpha_i}\cX_{[\alpha_i]}(c_i-d_j)\big)^{\br}(\beta_j)=0.
$$
Since $L_{\alpha_i}$ and $\br(\beta_j)$ are additive on $\mathbb F[z]$, the latter equality is equivalent to \eqref{4.25hab}. Again, by Proposition \ref{P:nhoma},
all $f\in  P_{n+k}(\mathbb F)$ satisfying conditions \eqref{1.18}, \eqref{1.19} are given by the formula \eqref{brap1} where $\psi_{ij}$ is any solution 
to the respective Sylvester equation \eqref{5.3b}. By Proposition \ref{P:klj}, $\psi_{ij}=\Psi_{\alpha_i,\beta_j}(c_i-d_j)+\varphi_{ij}$ where 
$\alpha_i\varphi_{ij}=\varphi_{ij} \beta_j$. Combining the latter representations with \eqref{brap1} gives \eqref{brap1}. The third sum on the right side
is taken over all conjugate pairs $\alpha_i\sim\beta_j$ as for non-conjugate pairs $\alpha_i\not\sim\beta_j$, we have $\varphi_{ij}=0$, by part (1) in Proposition  \ref{P:klj}.
\end{proof}
\begin{corollary}
Under the assumptions of Theorem \ref{T:4.1}, a  polynomial $g\in\mathbb P_{n+k}(\mathbb F)$
satisfies conditions $g^{\bl}\vert_{\Lambda}=0$ and $g^{\br}\vert_\Omega=0$ if and only if it is of the form
\begin{equation}
g(z)=P_{\Lambda,\boldsymbol\ell}(z)\cdot\sum_{\alpha_i\sim\beta_j}p_i^{\bl}(\alpha_i)^{-1}\varphi_{ij}q_j^{\br}(\beta_j)^{-1}q_j(z)
\label{ap22}
\end{equation}
where  $p_i=P_{\Lambda\backslash\{\alpha_i\},\boldsymbol\ell}$, $q_j=P_{\Omega\backslash\{\beta_j\},{\bf r}}$ and where $\varphi_{ij}$ is any
intertwiner of $\alpha_i$ and $\beta_j$.
\label{C:4.2}
\end{corollary}
The formula \eqref{ap22} follows upon letting $c_i=d_j=0$ for all $i,j$ in \eqref{brap1}. We next observe from the division algorithm that 
upon adding the term $P_{\Lambda,\boldsymbol\ell}hP_{\Omega,{\bf r}}$ on the right side of \eqref{ap22}
and letting $h$ to run through $\mathbb F[z]$, leads to a parametrization of the set of all polynomials $f\in\mathbb F[z]$ solving the homogeneous problem
\eqref{1.18}, \eqref{1.19}, that is, the intersection
of two (left and right) ideals $\langle P_{\Delta,{\boldsymbol\ell}}\rangle_{\bf r}\cap\langle P_{\Delta,{\bf r}}\rangle_{\boldsymbol\ell}$
(a quasi-ideal) of $\mathbb F[z]$, in the terminology of \cite{stein2}). 
\subsection{Two-sided $P$-independence} The property of an algebraic set $\Delta\subset \mathbb F$ to be left (right) $P$-independent 
can be characterized as follows (see the proof Theorem \ref{T:1.1}): {\em there is no nonzero polynomial 
$f\in P_{_{|\Delta|}}(\mathbb F)$ such that 
$f^{\bl}\vert_{\Delta}=0$ (respectively, $f^{\br}\vert_{\Delta}=0$)}. Combining the latter characterizations, we say that 
\begin{definition}
The pair $(\Delta_{\boldsymbol\ell},\Delta_{\bf r})$ consisting of two algebraic sets $\Delta_{\boldsymbol\ell}$ and 
$\Delta_{\bf r}$
is $P$-independent if there is no nonzero  $g\in P_{_{|\Delta_{\boldsymbol\ell}|+|\Delta_{\bf r}|}}(\mathbb F)$
such that
\begin{equation}
g^{\bl}\vert_{\Delta_{\boldsymbol\ell}}=0\quad\mbox{and}\quad g^{\br}\vert_{\Delta_{\bf r}}=0.
\label{opana}
\end{equation}
\noindent
{\rm Since the minimal polynomials $P_{\Lambda,\boldsymbol\ell}$, $P_{\Omega, {\bf r}}$ of (algebraic) sets $\Delta_{\boldsymbol\ell}$ ,
$\Delta_{\bf r}$ satisfy inequalities $\deg P_{\Delta_{\boldsymbol\ell},\boldsymbol\ell}\le |\Delta_{\boldsymbol\ell}|$, 
 $\deg P_{\Delta_{\bf r},{\bf r}}\le |\Delta_{\bf r}|$, whereas their product
$g=P_{\Delta_{\boldsymbol\ell},\boldsymbol\ell}\cdot P_{\Delta_{\bf r}, {\bf r}}$ satisfies conditions \eqref{opana}, we conclude 
(by Definition \ref{D:1.1}) that if the pair  $(\Delta_{\boldsymbol\ell},\Delta_{\bf r})$ is $P$-independent (and hence, 
$\deg(P_{\Delta_{\boldsymbol\ell},\boldsymbol\ell}\cdot P_{\Delta_{\bf r},{\bf r}})\ge 
|\Delta_{\boldsymbol\ell}|+|\Delta_{\bf r}|$), then $\Delta_{\boldsymbol\ell}$ and $\Delta_{\bf r}$ are respectively, 
left and right $P$-independent. In the case where at least one of them is algebraic over $Z_{\mathbb F}$, we can say more. 
Given a set $\Delta\subset\mathbb F$ we will denote by $\left[\Delta\right]:=
\bigcup_{\alpha\in\Delta}[\alpha]$ the minimal superset of $\Delta$ closed under conjugation.}
\label{D:1}
\end{definition}
\begin{proposition}
Let us assume that $\Delta_{\boldsymbol\ell}$ is algebraic over $Z_{\mathbb F}$. Then the pair $(\Delta_{\boldsymbol\ell},\Delta_{\bf r})$
is $P$-independent if and only if $\Delta_{\boldsymbol\ell}$ is left $P$-independent, $\Delta_{\bf r}$ is 
right $P$-independent, and $\left[\Delta_{\boldsymbol\ell}\right]\cap \left[\Delta_{\bf r}\right]=\varnothing$.
\label{P:ap12}
\end{proposition}
\begin{proof}
As we have already observed, if $(\Delta_{\boldsymbol\ell},\Delta_{\bf r})$ is $P$-independent, then 
$\Delta_{\boldsymbol\ell}$ and $\Delta_{\bf r}$ are left and right $P$-independent and therefore, 
contain finitely many elements. Hence we may let $\Delta_{\boldsymbol\ell}=\Lambda$ and $\Delta_{\bf r}=\Omega$ as in \eqref{1.7}. 
It remains to show that under the assumptions as in Theorem \ref{T:4.1}, the pair $(\Lambda,\Omega)$ is $P$-independent if and only if 
 $\left[\Lambda\right]\cap \left[\Omega\right]=\varnothing$. The latter follows by Corollary \ref{C:4.2}. Indeed, the formula 
\eqref{ap22} produces all polynomials $g\in P_{n+k}(\mathbb F)$ satisfying conditions \eqref{opana}. By Definition \ref{D:1},
the $(\Lambda,\Omega)$ is $P$-independent if and only if any $g$ of the form \eqref{ap22} is the zero polynomial, which means that
the only $x\in\mathbb F$ subject to $\alpha_i x=x\beta_j$ is $x=0$, i.e., that 
$\alpha_i\not\sim\beta_j$ for all $\alpha_i\in\Lambda$ and $\beta_j\in\Omega$.
\end{proof}
The next statement can be considered as a two-sided analog of Theorem  \ref{T:1.1}.
\begin{theorem}
Given two sets $\Lambda$ and $\Omega$ as in \eqref{1.7}, let us assume that $\Lambda$ is algebraic over $Z_{\mathbb F}$. Then
the problem \eqref{1.18}, \eqref{1.19} has a solution in $P_{n+k}(\mathbb F)$ for any $c_i,d_j\in \mathbb F$ if and only if the pair
$(\Lambda,\Omega)$ is $P$-independent.
In this case, a unique $f\in P_{n+k}(\mathbb F)$ subject to conditions \eqref{1.18}, \eqref{1.19} is given by 
\begin{equation}
f(z)=\sum_{i=1}^n P_{{\Lambda}\backslash\{\alpha_i\},{\boldsymbol\ell}}(z)\cdot \rho_i\cdot
P_{\Omega,{\bf r}}(z)+\sum_{j=1}^k P_{\Lambda,{\boldsymbol\ell}}(z)\cdot \gamma_j\cdot P_{\Omega\backslash\{\beta_j\},{\bf r}}(z)
\label{ma1}
\end{equation}
where the elements $\rho_i,\gamma_j\in\mathbb F$ are defined by
\begin{align}
\rho_i&=-\sum_{j=1}^k P_{{\Lambda}\backslash\{\alpha_i\},{\boldsymbol\ell}}^{\bl}(\alpha_i)^{-1}
\Psi_{\alpha_i,\beta_j}(c_i)
P_{\Omega\backslash\{\beta_j\},{\bf r}}^{\br}(\beta_j)^{-1},\label{ma2}\\
\gamma_j&=\sum_{i=1}^n  P_{{\Lambda}\backslash\{\alpha_i\},{\boldsymbol\ell}}^{\bl}(\alpha_i)^{-1}
\Psi_{\alpha_i,\beta_j}(d_j)
P_{\Omega\backslash\{\beta_j\},{\bf r}}^{\br}(\beta_j)^{-1},
\label{ma3}
\end{align}
whereas $\Psi_{\alpha_i,\beta_j}(c_i)$ and $\Psi_{\alpha_i,\beta_j}(d_j)$ are defined via the top formula in \eqref{combo}.
\label{T:ap2}
\end{theorem}
\begin{proof}
If the problem \eqref{1.18}, \eqref{1.19} has a solution for any choice of left and right target values, then $\Lambda$ is left $P$-independent
and $\Omega$ is  right $P$-independent (by Theorem \ref{T:1.1}). To complete the proof of the "only if" statement,
it remains (due to Proposition \ref{P:ap12}) to show that $\left[\Lambda\right]\cap \left[\Omega\right]=\varnothing$.
To argue via contradiction, let us assume that $\alpha_i\sim\beta_j$ (for some $i,j$). Then, by condition \eqref{4.25hab} in Theorem \ref{T:4.1}, we have
$$
\big(L_{\alpha_i}\cX_{[\alpha_i]}c\big)^{\br}(\beta_j)=\big(L_{\alpha_i}\cX_{[\alpha_i]}d\big)^{\br}(\beta_j)\quad\mbox{for any}\quad c,d\in\mathbb F.
$$
Letting $c=0$, we then conclude, by formula \eqref{2.7}, that
$$
0=\big(L_{\alpha_i}\cX_{[\alpha_i]}d\big)^{\br}(\beta_j)=\big(L_{\alpha_i}\cX_{[\alpha_i]}\big)^{\br}(d^{-1}\beta_j d)\cdot d\quad\mbox{for any}\quad d\in\mathbb F.
$$
The latter means that the polynomial $L_{\alpha_i}\cX_{[\alpha_i]}$ takes zero right value at any element in the conjugacy class $\alpha_i]$ and hence,
belongs to the ideal $\langle \cX_{\alpha_i}\rangle$, which is impossible, as $\deg L_{\alpha_i}\cX_{[\alpha_i]}<\deg \cX_{[\alpha_i]}$ and
$L_{\alpha_i}\cX_{[\alpha_i]}\not\equiv 0$. This completes the proof of the "only if" part. The converse implication follows from Theorems
\ref{T:1.1} and \ref{T:4.1}. 

\smallskip

A unique  low-degree solution to the problem \eqref{1.18}, \eqref{1.19} is given by the formula \eqref{brap1}, which, as will now show, 
can be written in the form \eqref{ma1}. To this end, we first observe that for each $j\in\{1,\ldots,k\}$,
Theorem \ref{T:4.1} applies to the interpolation problem 
\begin{equation}
f^{\br}(\beta_j)=d_t,\quad f^{\br}\vert_{\Omega\backslash\{\beta_j\}}=0,\quad f^{\bl}\vert_{\Lambda}=0.
\label{ma4}
\end{equation}
With the target values as above (that is, with $c_i=d_t=0$ for all $i$ and $t\neq j$), the formula \eqref{combo} gives 
$\Psi_{\alpha_i,\beta_t}(c_i-d_t)=0$ for all $t\neq j$. Hence, the formula \eqref{brap1} takes the form     
\begin{equation}
f_{{\bf r},j}(z)=P_{\Lambda,{\boldsymbol\ell}}(z)\cdot \gamma_j\cdot P_{\Omega\backslash\{\beta_j\},{\bf r}}(z),
\label{5.15}
\end{equation}
with $\gamma_j$ defined as in \eqref{ma3}. By Theorem \ref{T:4.1}, $f_{{\bf r},j}$ is a unique polynomial in $P_{n+k}(\mathbb F)$
satisfying conditions \eqref{ma4}. Similarly, by applying Theorem \ref{T:4.1} to the interpolation problem
\begin{equation}
f^{\bl}(\alpha_i)=c_i,\quad f^{\bl}\vert_{\Lambda\backslash\{\alpha_i\}}=0,\quad f^{\br}\vert_{\Omega}=0,
\label{ma5}
\end{equation}
and adapting the formula \eqref{ap1a} to the present case, we conclude that a unique polynomial in $P_{n+k}(\mathbb F)$
subject to conditions \eqref{ma5} is given by the formula 
\begin{equation}
f_{{\boldsymbol \ell},i}(z)=P_{{\Lambda}\backslash\{\alpha_i\},{\boldsymbol\ell}}(z)\cdot \rho_i\cdot P_{\Omega,{\bf r}}(z),
\label{5.19}
\end{equation}
with $\rho_i$ defined as in \eqref{ma2}. Combining \eqref{ma4} and \eqref{ma5} we see that the formula 
\eqref{ma1} defines a polynomial $f\in P_{n+k}(\mathbb F)$ satisfying conditions \eqref{1.18}, \eqref{1.19}. 
By the uniqueness part in Theorem \ref{T:4.1}), this is the same polynomial as in \eqref{brap1}.
\end{proof}
\begin{remark}
{\rm A decomposition of a low-degree solution to the Lagrange problem into the sum of ``elementary"
polynomials each of which satisfies one requisite interpolation condition and equals zero at all other interpolation nodes,
is commonly termed as the Lagrange interpolation formula. For this reason, the formula \eqref{ma1} (rather than \eqref{brap1} or \eqref{ap1a}) can be referred
to as to the {\em two-sided Lagrange interpolation formula}. Other examples (commutative, left, right) are provided by respective formulas
\eqref{1.2}, \eqref{1.16}, \eqref{1.16r}.} 
\label{R:bal}
\end{remark}
\subsection{Interpolation within an algebraic conjugacy class} Let us assume that the sets $\Lambda$ and 
$\Omega$ in \eqref{1.7} are respectively, left and right $P$-independent, and moreover, that they are contained in the same algebraic conjugacy class $V$. 
By Theorem \ref{T:4.1}, the problem \eqref{1.18}, \eqref{1.19} has a solution if and only equalities \eqref{4.25hab} hold for all $i\in\{1,\ldots,n\}$ 
and $j\in\{1,\ldots,k\}$, in which case all $f\in P_{n+k}(\mathbb F)$ solving the problem are given by the formula \eqref{brap1}, where 
$\Psi_{\alpha_i,\beta_j}(c_i-d_j)$ is defined by the bottom formula in \eqref{combo} for all $i,j$. 
\begin{remark}
{\rm In the present case, the parametrization formula \eqref{brap1} cannot be written in the form of the Lagrange interpolation formula \eqref{ma1} 
since the polynomials $f_{{\bf r},j}$ and $f_{{\boldsymbol\ell},i}$ solving ``elementary" interpolation problems \eqref{ma4} and \eqref{ma5} may not exist.
By the general criterion \eqref{4.25ha}, these polynomials do exist if and only if
\begin{equation}
\big(L_{\alpha_i}\cX_{_{V}} c_i\big)^{\br}(\beta_j)=\big(L_{\alpha_i}\cX_{_{V}}d_j\big)^{\br}(\beta_j)=0\quad
\mbox{for all} \quad \alpha_i\in\Lambda, \; \beta_j\in\Omega,
\label{jun4}
\end{equation}
that is, if and only if the elements $c_i\beta_j c_i^{-1}$ and $d_j\beta_j d_j^{-1}$ are
right zeros of the polynomial $L_{\alpha_i}\cX_{_{V}}$ for all $i,j$ such that $c_i\neq 0$ and $d_j\neq 0$.
(note that if one of the conditions \eqref{jun4} holds, then all other conditions hold as well).
In this case, a particular solution $f\in P_{n+k}(\mathbb F)$ to the problem \eqref{1.18}, \eqref{1.19}
is given by the formulas \eqref{ma1}-\eqref{ma3}, where $\Psi_{\alpha_i,\beta_j}(c_i)$ and $\Psi_{\alpha_i,\beta_j}(d_j)$ are defined via the bottom 
formula in \eqref{combo} rather the top one.}
\label{R:jun2}
\end{remark}
Equalities \eqref{4.25hab} guarantee the consistency of interpolation conditions \eqref{1.18} and \eqref{1.19}. Via these equalities,
the left target values impose certain restrictions on the right ones (and vice versa) and in general, none of them can be eliminated as redundant.
The case where $\Lambda$ (or $\Omega$) is a $P$-basis for $V$ is more rigid.
\begin{proposition}
If $\Lambda$ is a left $P$-basis for $V$, then {\rm (1)} $d_j$'s are uniquely determined from \eqref{4.25hab} and {\rm (2)} any polynomial $f\in\mathbb F[z]$ satisfying 
left conditions \eqref{1.18}, automatically satisfies right conditions \eqref{1.19}. Similar statements hold if $\Omega$ is a right $P$-basis for $V$.
\label{P:aut}
\end{proposition}
\begin{proof} By part (2) in Proposition \ref{P:klj}, relations
\eqref{4.25hab} guarantee the existence of elements $\psi_{ij}\in\mathbb F$ subject to equations \eqref{5.3b}. Multiplying both sides of \eqref{5.3c}
by $p_i^{\bl}(\alpha_i)$ on the right and taking into account the definition
of $\widetilde{\alpha}_i$ in \eqref{5.3c}, we get 
$$
p_i\cdot p_i^{\bl}(\alpha_i)^{-1}\bp_{\alpha_i}=P_{\Lambda,{\boldsymbol\ell}}\cdot p_i^{\bl}(\alpha_i)^{-1}\quad\mbox{for}\quad i=1,\ldots, n.
$$
Since  in the present case, $P_{\Lambda,{\boldsymbol\ell}}=\cX_{_{V}}\in Z_{\mathbb F}[z]$ and $\beta_j\in V$, we have
$$
(p_i\cdot  p_i^{\bl}(\alpha_i)^{-1}\psi_{ij}\big)^{\br}(\beta_j)=p_i^{\bl}(\alpha_i)^{-1}\psi_{ij}\cdot\cX_{_{V}}(\beta_j)=0.
$$
Making use of the latter equalities along with \eqref{5.3d} and \eqref{4.25hab}, we get
\begin{align}
\sum_{i=1}^n \big( p_i\cdot p_i^{\bl}(\alpha_i)^{-1}
c_i\big)^{\br}(\beta_j)
&=\sum_{i=1}^n \big(p_i \cdot p_i^{\bl}(\alpha_i)^{-1}
(d_j+\psi_{ij}\cdot\bp_{\beta_j}-\bp_{\alpha_i}\cdot\psi_{ij}\big)^{\br}(\beta_j)\notag\\
&=d_j-\sum_{i=1}^n \big( p_i\cdot  p_i^{\bl}(\alpha_i)^{-1}\psi_{ij}\big)^{\br}(\beta_j)=d_j.\label{aut1}
\end{align}
By the formula \eqref{6.3r} in Lemma \ref{L:ext} (with $m=n$, $\gamma_i=\alpha_i$ and 
$f^{\bl}(\alpha_i)=c_i$ for $i=1,\ldots,n$), if $f$ satisfies conditions \eqref{1.18}, then $f^{\br}(\beta_j)$ is defined by the expression 
on the left side of \eqref{aut1}, i.e., conditions $f^{\br}(\beta_j)=d_j$ are satisfied automatically.
\end{proof}
Thus, if $\Lambda$ is a left $P$-basis for $V$, conditions \eqref{1.19} can be dismissed leaving us with a left-sided problem \eqref{1.18}. 
More generally, if $\Lambda$ {\em contains} a left $P$-basis for a conjugacy class $V$, then all right sided conditions at $\beta_j\in V$ can be 
dismissed as redundant. Similar observations apply to the case where $\Omega$ contains a right $P$-basis for some conjugacy class.
Note that without the above dismissal, one can still use the parametrization formula \eqref{brap1}, which now takes the form 
$$
f(z)=\sum_{i=1}^n p_i(z)p_i^{\bl}(\alpha_i)^{-1}c_i+\cX_{_{V}}(z)\cdot h(z),\qquad h\in P_{k}(\mathbb F),
$$
and produces all polynomials $f\in P_{n+k}(\mathbb F)$ subject to the left conditions \eqref{1.18}. 
\subsection{Generalized Lagrange interpolation formula} 
As we observed in Remark \ref{R:jun2}, the Lagrange interpolation formula \eqref{ma1} may not exist if $[\Lambda]\cap[\Omega]\neq\varnothing$.
However, it is possible to decompose a low-degree solution to the problem into the sum of ``elementary" polynomials each one of which satisfies 
the required interpolation conditions within one conjugacy class and vanishes at all interpolation nodes outside this class. 
To be more precise, let $V_1,\ldots,V_m$ be all conjugacy classes in $\mathbb F$ 
having non-empty intersection with both $\Lambda$ and $\Omega$. Letting
$$
\Lambda_0=\Lambda\backslash \left[\Omega\right],\quad \Omega_0=\Omega\backslash \left[\Lambda\right],\quad
\Lambda_s:=V_s\cap \Lambda,\quad \Omega_s:=V_s\cap \Omega\quad (s=1,\ldots,m)
$$
we arrive at the partitions $\Lambda=\bigcup_{s=0}^m\Lambda_s\;$ and $\; \Omega=\bigcup_{s=0}^m\Omega_s$ of the sets $\Lambda$ and $\Omega$. 
By the {\em generalized Lagrange formula}, we mean 
a representation of a low-degree solution $f\in P_{n+k}(\mathbb F)$ to the problem \eqref{1.18}, \eqref{1.19} in the form
\begin{equation}
f=\sum_{\alpha_i\in\Lambda_0} P_{{\Lambda}\backslash\{\alpha_i\},{\boldsymbol\ell}}\cdot \rho_i\cdot
P_{\Omega,{\bf r}}+\sum_{\beta_j\in\Omega_0} P_{\Lambda,{\boldsymbol\ell}}\cdot \gamma_j\cdot P_{\Omega\backslash\{\beta_j\},{\bf r}}
+\sum_{s=1}^m P_{\Lambda\backslash \Lambda_s,\boldsymbol\ell}\cdot g_s \cdot P_{\Omega\backslash \Omega_s,{\bf r}}
\label{ma54}
\end{equation}
for some $\rho_i,\gamma_j\in\mathbb F$ and $g_s\in P_{|\Lambda_s|+|\Omega_s|}(\mathbb F)$. Note that in case $[\Lambda]\cap[\Omega]=\varnothing$, the
formula \eqref{ma54} amounts to \eqref{ma1}. The polynomials 
\begin{equation}
f_{{\boldsymbol \ell},i}=P_{{\Lambda}\backslash\{\alpha_i\},{\boldsymbol\ell}}\cdot \rho_i\cdot P_{\Omega,{\bf r}}, \; \; 
f_{{\bf r},j}=P_{\Lambda,{\boldsymbol\ell}}\cdot \gamma_j\cdot P_{\Omega\backslash\{\beta_j\},{\bf r}}, \; \; 
f_{_{V_s}}=P_{\Lambda\backslash \Lambda_s,\boldsymbol\ell}\cdot g_s \cdot P_{\Omega\backslash \Omega_s,{\bf r}}
\label{ma56}
\end{equation}
on the right side of \eqref{ma54} clearly satisfy the following homogeneous conditions 
\begin{align}
&f_{{\boldsymbol\ell},i}^{\bl}(\alpha)=0, \quad f_{{\boldsymbol\ell},i}^{\bl}(\beta)=0\quad  
\mbox{for all}\; \; \alpha\in\Lambda\backslash\{\alpha_i\}, \; \beta\in\Omega,\label{5.20u}\\
&f_{{\bf r},j}^{\bl}(\alpha)=0, \quad f_{{\bf r},j}^{\br}(\beta)=0\quad
\mbox{for all}\; \; \alpha\in\Lambda, \; \beta\in\Omega\backslash\{\beta_j\},\label{5.20v}\\
&f_{_{V_s}}^{\bl}(\alpha)=0, \quad f_{_{V_s}}^{\br}(\beta)=0\quad 
\mbox{for all}\; \; \alpha\in\Lambda\backslash\Lambda_s, \; \beta\in\Omega\backslash\Omega_s.\notag
\end{align}
Therefore, for $f$ of the form \eqref{ma54}, we have 
$$
f^{\bl}(\alpha_i)=f_{{\boldsymbol\ell},i}^{\bl}(\alpha_i) \; \; \mbox{for} \; \alpha_i\in\Lambda_0, \quad 
f^{\br}(\beta_j)=f_{{\bf r},j}^{\br}(\beta_j)\; \; \mbox{for} \; \beta_j\in\Omega_0,\quad\mbox{and}
$$
$$
f^{\bl}(\alpha_i)=f_{_{V_s}}^{\bl}(\alpha_i),\quad f^{\br}(\beta_j)=f_{_{V_s}}^{\br}(\beta_j)\quad\mbox{for}\quad \alpha_i,\beta_j\in V_s; \; \; s=1,\ldots,m.
$$
To make sure that $f$ of the form \eqref{ma54} satisfies interpolation conditions \eqref{1.18}, \eqref{1.19}, it remains to appropriately 
specify the elements $\rho_i$, $\gamma_j$ and the polynomials $g_s$ in \eqref{ma54}. The elements  $\rho_i$, $\gamma_j$ are defined uniquely
by formulas \eqref{5.18} and \eqref{5.14} below (which are alternative to formulas \eqref{ma2} and \eqref{ma3}), whereas $g_s$ is any polynomial 
in $P_{|\Lambda_s|+|\Omega_s|}(\mathbb F)$ solving a two-sided problem \eqref{5.20y} below (at modified interpolation nodes \eqref{5.9u} in $V_s$ and 
modified target values \eqref{5.22}, \eqref{5.21}). Here we will use an approach based on left and right $\lambda$-transforms introduced 
and studied in \cite{lamler13, lamler2}. We assume that all interpolation nodes are algebraic and lay out some extra notation.

\smallskip

For a polynomial $g\in\mathbb F[z]$, we denote by $\mathfrak D_g\in Z_{\mathbb F}[z]$
the greatest central divisor of $g$, i.e., the generator of the smallest two-sided ideal containing 
$\langle g\rangle_{\boldsymbol\ell}$ or $\langle g\rangle_{\bf r}$. We will denote by $\mathfrak Q_g$ a unique polynomial
such that 
$$
g=\mathfrak D_g \mathfrak Q_g=\mathfrak Q_g\mathfrak D_g.
$$
A polynomial $g$ is called {\em bounded}
if there exists a central multiple of $g$, in which case we will denote by $\mathfrak M_g\in Z_{\mathbb F}[z]$ the {\em least 
central multiple} of $g$ (the generator of the largest two-sided ideal contained in $\langle g\rangle_{\boldsymbol\ell}$ 
or in $\langle g\rangle_{\bf r}$). We will denote by $g^\diamondsuit$  a unique polynomial
such that  
$$\mathfrak M_g=gg^\diamondsuit=g^\diamondsuit g.$$ 
From these definitions, it is readily seen that 
\begin{equation}
g^\diamondsuit =(\mathfrak Q_g)^\diamondsuit,\quad (g^{\diamondsuit})^{\diamondsuit}=\mathfrak Q_g,\quad 
\mathfrak D_g \mathfrak M_{g^\diamondsuit}=\mathfrak D_g \mathfrak Q_g(\mathfrak Q_g)^\diamondsuit=\mathfrak M_{g}.
\label{4.38}
\end{equation}
Following \cite{lamler2}, we associate with a given polynomial $h\in\mathbb F[z]$ and an element $\beta\in\mathbb F$ two self-maps of $\mathbb F\backslash\{0\}$
(left and right $\lambda_{h,\beta}$-transforms)
$$
\delta\mapsto (\delta h)^{\bl}(\beta)=\delta\cdot h^{\bl}(\delta^{-1}\beta\delta)\quad\mbox{and}\quad
\delta\mapsto (h\delta)^{\br}(\beta)=h^{\br}(\delta\beta\delta^{-1})\cdot\delta.
$$
The formulas for inverse transformations are presented in the next lemma.
\begin{lemma}
Given $\beta\in\mathbb F$ and bounded $h\in\mathbb F[z]$ such that $\mathfrak M_h(\beta)\neq 0$,
\begin{align}
&d=(h\delta)^{\br}(\beta)\; \; \Leftrightarrow \; \;
\delta=(h^{\diamondsuit}d)^{\br}(\beta)\cdot\mathfrak M_h(\beta)^{-1};
\label{5.12} \\
&d=(\delta h)^{\bl}(\beta) \; \; \Leftrightarrow \; \;
\delta=\mathfrak M_h(\beta)^{-1}\cdot (dh^{\diamondsuit})^{\bl}(\beta)
\label{5.13}
\end{align}
for any $d,\delta\in\mathbb F\backslash\{0\}$.
\label{L:7.1}
\end{lemma}
\begin{proof}
If $d=(h\delta)^{\br}(\beta)$, then by the formula \eqref{2.7}
(with $f=h^\diamondsuit$ and $g=h\delta$) we have
$$
(h^{\diamondsuit}d)^{\br}(\beta)=(h^\diamondsuit h\delta)^{\br}(\beta)=\delta\cdot\mathfrak M_h(\beta),
$$
where the second equality holds since $\mathfrak M_h\in Z_{\mathbb F}[z]$. Since $\mathfrak M_h(\beta)\neq 0$, the latter formula
implies the formula for $\delta$ in \eqref{5.12} proving the
implication $\Rightarrow$ in \eqref{5.12}.  For the reverse implication, write the second equality in \eqref{5.12} equivalently as
$$
\delta\cdot\mathfrak M_h(\beta)=(h^{\diamondsuit}d)^{\br}(\beta).
$$
We then apply the implication $\Rightarrow$ (just proven) to the latter equality (i.e.,
to $h^\diamondsuit$, $d$ and $\delta\cdot\mathfrak M_h(\beta)$ rather than $h$, $\delta$ and $d$) and then make use 
of the second and the third relations in \eqref{4.38} to get
\begin{equation}
d=(h^{\diamondsuit\diamondsuit}\delta)^{\br}(\beta)\cdot \mathfrak M_h(\beta)\cdot\mathfrak M_{h^\diamondsuit}(\beta)^{-1}
=(\mathfrak Q_h \delta)^{\br}(\beta)\cdot\mathfrak D_h(\beta).
\label{po21}
\end{equation}
Taking into account that $\mathfrak D_h\in Z_{\mathbb F}[z]$ and that $\mathfrak D_h(\beta)$ commutes with $\beta$,
we use the formula \eqref{2.7} to compute
\begin{align*}
(h\delta)^{\br}(\beta)=(\mathfrak Q_h \mathfrak D_h\delta)^{\br}(\beta)
&=\mathfrak Q_h^{\br}\big(\delta \mathfrak D_h(\beta)\beta \mathfrak D_h(\beta)^{-1}\delta^{-1}\big)\cdot \delta \cdot \mathfrak D_h(\beta)\\
&=\mathfrak Q_h^{\br}\big(\delta\beta\delta^{-1}\big)\cdot \delta \cdot \mathfrak D_h(\beta)=(\mathfrak Q_h\delta)^{\br}(\beta)
\cdot \mathfrak D_h(\beta),
\end{align*}
which together with \eqref{po21} implies $d=(h\delta)^{\br}(\beta)$, thus
completing the proof of \eqref{5.12}. The equivalence \eqref{5.13} is verified similarly.
\end{proof}
\noindent
We next apply Lemma \ref{L:7.1} to get the formulas for the elements $\rho_i, \gamma_j$ and to specify polynomials $g_s$ in \eqref{ma54}.
\begin{lemma}
{\rm (1)} If $\alpha_i\in\Lambda_0=\Lambda\backslash\left[\Omega\right]$, then the polynomial  
$f_{{\boldsymbol \ell},i}=P_{{\Lambda}\backslash\{\alpha_i\},{\boldsymbol\ell}}\rho_i P_{\Omega,{\bf r}}$ satisfies 
$f_{{\boldsymbol \ell},i}^{\bl}(\alpha_i)=c_i$ if and only if
\begin{equation}
\rho_i=\left\{\begin{array}{ccc}P_{\Lambda\backslash\{\alpha_i\},{\boldsymbol\ell}}^{\bl}(\alpha_i)^{-1}
\cdot \mathfrak M_{P_{\Omega,{\bf r}}}(\alpha_i)^{-1}\cdot
(c_i P_{\Omega,{\bf r}}^{\diamondsuit})^{\bl}(\alpha_i), & \mbox{if}& c_i\neq 0,\\
0,& \mbox{if}& c_i=0.\end{array}\right.\label{5.18}
\end{equation}
{\rm (2)} If $\beta_j\in \Omega_0=\Omega\backslash\left[\Lambda\right]$, then the polynomial
$f_{{\bf r},j}=P_{\Lambda,{\boldsymbol\ell}}\gamma_jP_{\Omega\backslash\{\beta_j\},{\bf r}}$
satisfies $f_{{\bf r},j}^{\br}(\beta_j)=d_j$ if and only if
\begin{equation}
\gamma_j=\left\{\begin{array}{ccc} (P_{\Lambda,{\boldsymbol\ell}}^{\diamondsuit}d_j)^{\br}(\beta_j)\cdot
\mathfrak M_{P_{\Lambda,{\boldsymbol\ell}}}(\beta_j)^{-1}\cdot P_{\Omega\backslash\{\beta_j\},{\bf r}}^{\br}(\beta_j)^{-1},&\mbox{if}& d_j\neq 0,\\
0,&\mbox{if}& d_j=0.\end{array}\right.
\label{5.14}
\end{equation}
\label{L:7.2}
\end{lemma}
\begin{proof} The case $c_i=\rho_i=0$ is obvious. 
If $c_i\neq 0$, we apply the equivalence \eqref{5.13} to $h=P_{\Omega,{\bf r}}$, 
$\delta=P_{\Lambda\backslash\{\alpha_i\},{\boldsymbol\ell}}\cdot\rho_i$, $d=c_i$ and $\beta=\alpha_i$ to conclude that 
$c_i=f_{{\boldsymbol\ell},i}^{\bl}(\alpha_i)=( P_{{\Lambda}\backslash\{\alpha_i\},{\boldsymbol\ell}}\cdot \rho_i\cdot
P_{\Omega,{\bf r}})^{\bl}(\alpha_i)$
if and only if 
$$
P_{\Lambda\backslash\{\alpha_i\},{\boldsymbol\ell}}\cdot\rho_i=
\mathfrak M_{P_{\Omega,{\bf r}}}(\alpha_i)^{-1}\cdot (c_iP_{\Omega,{\bf r}}^{\diamondsuit})^{\bl}(\alpha_i).
$$
The latter is equivalent to the top formula in \eqref{5.18}. The proof of part (2) relies on the equivalence \eqref{5.12} and is quite similar.
\end{proof}
\begin{lemma}
Let $\alpha_i$ and $\beta_j$ belong to the conjugacy class $V_s$.
A polynomial $f_{_{V_s}}=P_{\Lambda\backslash \Lambda_s,\boldsymbol\ell}\cdot g_s \cdot P_{\Omega\backslash \Omega_s,{\bf r}}$ satisfies conditions
\begin{equation}
f_{_{V_s}}^{\bl}(\alpha_i)=c_i\quad\mbox{and}\quad f_{_{V_s}}^{\br}(\beta_j)=d_j
\label{5.20uu}
\end{equation}
if and only if $g_s\in\mathbb F[z]$ is subject to
\begin{equation}
g_s^{\bl}(\widetilde{\alpha}_i)=\rho_i\quad\mbox{and}\quad g_s^{\br}(\widetilde{\beta}_j)=\gamma_j
\label{5.20y}
\end{equation}
where $\widetilde{\alpha}_i, \widetilde{\beta}_j\in V_s$ are defined by 
\begin{equation}
\widetilde{\alpha}_i=P_{\Lambda\backslash \Lambda_s,\boldsymbol\ell}^{\bl}(\alpha_i)^{-1}\cdot\alpha_i\cdot 
P_{\Lambda\backslash \Lambda_s,\boldsymbol\ell}^{\bl}(\alpha_i),\quad
\widetilde{\beta}_j=P_{\Omega\backslash \Omega_s,{\bf r}}^{\br}(\beta_j)\cdot\beta_j\cdot 
P_{\Omega\backslash \Omega_s,{\bf r}}^{\br}(\beta_j)^{-1},
\label{5.9u}
\end{equation}
and where the elements $\rho_i$ and $\gamma_j$ are given by (compare with \eqref{5.14} and \eqref{5.18})
\begin{equation}
\rho_i=\left\{\begin{array}{ccc}P_{\Lambda\backslash \Lambda_s,{\boldsymbol\ell}}^{\bl}(\alpha_i)^{-1}
\cdot \mathfrak M_{P_{\Omega\backslash \Omega_s,{\bf r}}}(\alpha_i)^{-1}\cdot 
(c_i P_{\Omega\backslash \Omega_s,{\bf r}}^{\diamondsuit})^{\bl}(\alpha_i ), & \mbox{if}& c_i\neq 0,\\
0,& \mbox{if}& c_i=0,\end{array}\right.
\label{5.22}
\end{equation}
\begin{equation}
\gamma_j=\left\{\begin{array}{ccc}(P_{\Lambda\backslash \Lambda_s,\boldsymbol\ell}^{\diamondsuit}d_j)^{\br}(\beta_j)\cdot
\mathfrak M_{P_{\Lambda\backslash \Lambda_s,\boldsymbol\ell}}(\beta_j)^{-1}\cdot P_{\Omega\backslash \Omega_s,{\bf r}}^{\br}(\beta_j)^{-1},
&\mbox{if}& d_j\neq 0,\\ 0,&\mbox{if}& d_j=0.\end{array}\right.
\label{5.21}
\end{equation}
\label{L:5.6}
\end{lemma}
\begin{proof}
Since the polynomials $P_{\Lambda\backslash \Lambda_s,\boldsymbol\ell}$ and $P_{\Omega\backslash \Omega_s,{\bf r}}$
have no zeros in $V_s$, their values at $\alpha_i,\beta_j\in V_i$ are not zeros, and the formulas \eqref{5.9u}, \eqref{5.22},
\eqref{5.21} make sense. We next verify that $d_j$ and $c_i$ are recovered from \eqref{5.21} and \eqref{5.22} by
\begin{equation}
d_j=\left(P_{{\Lambda}\backslash \Lambda_s,{\boldsymbol \ell}}\cdot\gamma_j\cdot P_{\Omega\backslash \Lambda_s,{\bf r}}\right)^{\br}(\beta_j),
\qquad c_i=\left(P_{\Lambda\backslash \Lambda_s,{\boldsymbol \ell}}\cdot\rho_i\cdot P_{\Omega\backslash \Lambda_s,{\bf
r}}\right)^{\bl}(\alpha_i).
\label{5.29}
\end{equation}
The trivial cases where $d_j=0$ and $c_i=0$ are clear. If $d_j\neq 0$, we have from \eqref{5.21},
$$
\gamma_j\cdot P_{\Omega\backslash \Omega_s,{\bf r}}^{\br}(\beta_j)=
(P_{\Lambda\backslash \Lambda_s,\boldsymbol\ell}^{\diamondsuit}d_j)^{\br}(\beta_j)\cdot
\mathfrak M_{P_{\Lambda\backslash \Lambda_s,\boldsymbol\ell}}(\beta_j)^{-1}
$$
and by implication $\; \Leftarrow\; $ in \eqref{5.12} and formula \eqref{2.7} we conclude
$$
d_j=\big(P_{\Lambda\backslash \Lambda_s,{\boldsymbol \ell}}\cdot\gamma_j\cdot P_{\Omega\backslash \Omega_s,{\bf r}}^{\br}(\beta_j)\big)^{\br}(\beta_j)\\
=\left(P_{\Lambda\backslash \Lambda_s,{\boldsymbol \ell}}\cdot\gamma_j\cdot P_{\Omega\backslash \Omega_s,{\bf r}}\right)^{\br}(\beta_j),
$$
which confirms the first equality in \eqref{5.29}. The second equality for $c_s\neq 0$ is verified in much the same way. On the other hand, for 
$f_{_{V_s}}$ defined as in \eqref{ma56}, we have, by the formulas \eqref{2.3b} and \eqref{2.7} and by the definition \eqref{5.9u},
of $\widetilde{\beta}_j$, 
\begin{align}
f_{_{V_s}}^{\br}(\beta_j)&=\big(P_{\Lambda\backslash \Lambda_s,\boldsymbol\ell}\cdot g_s \cdot 
P_{\Omega\backslash \Omega_s,{\bf r}}\big)^{\br}(\beta_j)\notag\\
&=\big(P_{\Lambda\backslash \Lambda_s,\boldsymbol\ell}\cdot (g_s \cdot P_{\Omega\backslash \Omega_s,{\bf r}})^{\br}(\beta_j)\big)^{\br}(\beta_j)\notag\\
&=\big(P_{\Lambda\backslash \Lambda_s,\boldsymbol\ell}\cdot g_s^{\br}(\widetilde{\beta}_t)\cdot P_{\Omega\backslash \Omega_s,{\bf r}}^{\br}(\beta_j)
\big)^{\br}(\beta_j)\notag\\
&=\big(P_{\Lambda\backslash \Lambda_s,\boldsymbol\ell}\cdot g_s^{\br}(\widetilde{\beta}_s)\cdot
P_{\Omega\backslash \Omega_s,{\bf r}}\big)^{\br}(\beta_j),\label{4.80}
\end{align}
and quite similarly,
\begin{equation}
f_{_{V_s}}^{\bl}(\alpha_i)=\big(P_{\Lambda\backslash \Lambda_s,\boldsymbol\ell}\cdot g_s \cdot 
P_{\Omega\backslash \Omega_s,{\bf r}}\big)^{\bl}(\alpha_i)=\big(P_{\Lambda\backslash \Lambda_s,\boldsymbol\ell}\cdot g_s^{\bl}(\widetilde{\alpha}_i)\cdot
P_{\Omega\backslash \Omega_s,{\bf r}}\big)^{\bl}(\alpha_i).\label{4.81}
\end{equation}
Comparing \eqref{4.80}, \eqref{4.81} with equalities \eqref{5.29} we conclude that $f_{_{V_s}}$ of the form \eqref{ma56} satisfies 
\eqref{5.20uu} if and only $g_s$ is subject to conditions \eqref{5.20y}.
\end{proof}
\begin{remark}
{\rm By Theorem \ref{T:4.1}, the existence of a polynomial $f_{_{V_s}}$ satisfying conditions \eqref{5.20uu} is equivalent to the equality
$$
\big(L_{\alpha_i} \mathcal X_{_{V_s}} d_j\big)^{\br}(\beta_j)=\big(L_{\alpha_i} \mathcal X_{_{V_s}} c_i\big)^{\br}(\beta_j),
$$
while the existence of a polynomial $g_s$ satisfying conditions \eqref{5.20y} is equivalent to 
$$
\big(L_{\widetilde{\alpha}_i} \mathcal X_{_{V_s}} \gamma_j\big)^{\br}(\widetilde{\beta}_j)=\big(L_{\widetilde{\alpha}_i} 
\mathcal X_{_{V_s}} \rho_i\big)^{\br}(\widetilde{\beta}_j).
$$
By Lemma \ref{L:5.6}, we conclude that the two latter equalities are equivalent.} 
\label{R;last}
\end{remark}
Lemma \ref{L:5.6} clarifies the choice of $g_s$ in the formula \eqref{ma54}. We consider {\em all} interpolation conditions in the original problem
\eqref{1.18}, \eqref{1.19} within the conjugacy class $V_s$ and then take $g_s$ to be any solution of the associated problem \eqref{5.20y}
(with equally many interpolation conditions within the same conjugacy class). Parametrization of all such $g_s$ can be obtained via general formula \eqref{brap1}
as explained in Section 3.2. Substituting these parametrizations for all $s=1,\ldots,m$ into \eqref{ma54} one can get a slightly more structured generalized 
Lagrange interpolation formula.

\bibliographystyle{amsplain}

\end{document}